\documentclass[11pt]{article}
\usepackage{amsrefs, amsmath, amsfonts, amssymb, amsthm, latexsym, hyperref, comment}
\usepackage[margin=1in]{geometry}
\usepackage{enumerate}
\usepackage{authblk}
\usepackage{xcolor}
\usepackage{csquotes}
\usepackage[symbol]{footmisc}

\numberwithin{equation}{section}
\theoremstyle{plain}
\newtheorem{thm}{Theorem}[section]

\newtheorem{thma}{Theorem}

\newtheorem{lem}{Lemma}[section]
\newtheorem{rmk}{Remark}
\newtheorem{prop}{Proposition}
\newtheorem{clm}[lem]{Claim}

\newtheorem{obs}[lem]{Observation}
\newtheorem*{obs*}{Observation}

\newtheorem*{cor*}{Corollary}

\newcommand{\ds}{\displaystyle}
\newcommand{\norm}[1]{\left\Vert#1\right\Vert}
\newcommand{\wh}{\widehat}
\newcommand{\sinc}{\mathrm{sinc}}
\newcommand{\sgn}{\mathrm{sgn}\,}

\newcommand{\FF}{{\mathcal F}}
\newcommand{\R}{\mathbb R}

\newcommand{\E}{\mathbb E}
\newcommand{\N}{\mathbb N}
\newcommand{\Z}{\mathbb Z}
\newcommand{\Pro}{\mathbb P}
\newcommand{\lm}{\lambda}
\newcommand{\si}{\sigma}
\newcommand{\g}{\gamma}
\newcommand{\al}{\alpha}
\newcommand{\ep}{\varepsilon}
\newcommand{\p}{\varphi}

\newcommand{\var}{\mathrm{var}\,}
\newcommand{\sprt}{\mathrm{sprt}\,}

\newcommand{\ind}{1{\hskip -2.5 pt}\hbox{I}}
\newcommand{\D}{\Delta}
\newcommand{\calN}{\mathcal{N}}

\makeatletter
\newcommand{\subjclass}[2][1991]{%
  \let\@oldtitle\@title%
  \gdef\@title{\@oldtitle\footnotetext{#1 \emph{Mathematics subject classification.} #2}}%
}
\newcommand{\keywords}[1]{%
  \let\@@oldtitle\@title%
  \gdef\@title{\@@oldtitle\footnotetext{\emph{Key words and phrases.} #1.}}%
}
\makeatother

\author[N. Feldheim, O. Feldheim \& S. Nitzan]{Naomi Feldheim, Ohad Feldheim and Shahaf Nitzan}
\title{Persistence of Gaussian stationary processes:\\ a spectral perspective}

\begin{document}

\date{}
\maketitle


\begin{abstract}
{We study the persistence probability of a centered stationary Gaussian process on 
$\mathbb Z$ or $\mathbb R$, that is,
its probability to remain positive for a long time. We describe the delicate interplay between this probability and the behavior of the spectral measure of the process near zero and infinity.}
\end{abstract}

\section{Introduction}\label{sec: intro}

\subsection{General Introduction} 
The persistence of a stochastic process $f$ above a certain level $\ell$, that is, the probability
that $f(t)>\ell$ for all $t$ in some large interval,
is a classical topic of study (see the recent surveys by Aurzada-Simon~\cite{AS} and Bray-Majumdar-Schehr~\cite{BMS}).
Here we investigate the persistence probability for the class of Gaussian stationary processes (GSP's) above the mean. This quantity
has been extensively studied since the 1950's, by Slepian \cite{Slep}, Newell-Rosenblatt \cite{NR} and many others, with old and new applications
in mathematical physics, engineering and other areas of probability \cites{BMS, DM1, DM2, ScMa}.
Nonetheless, until recently, good estimates of the persistence decay were known only for particular cases (e.g. \cites{ABMO, Slep}), and for families of processes with either summable or non-negative correlations. The state of the art in the latter case was recently achieved by Dembo-Mukherjee \cite{DM2}, who were able to determine the log persistence of non-negatively correlated GSP's up to a constant factor.

A few years ago, by introducing a spectral point of view, the first two authors
were able to provide general conditions under which the log persistence is bounded between two linear functions~\cite{FF}.
This extended a result by Antezana-Buckley-Marzo-Olsen for the \emph{sinc-kernel} process~\cite{ABMO},
and provided the first general result on persistence of {GSP's} which does not require summability or non-negativity of correlations.
However,
these tools alone were insufficient to provide answers to two long-standing questions formulated by Slepian in his well known 1962 paper \cite{Slep}:

 \begin{itemize}
 \item
 What are the possible asymptotic behaviors of the persistence probability of a GSP on large intervals?
 \item What features of the covariance function determine this behavior?
\end{itemize}


Spectral methods were recently used by Krishna-Krishnapur \cite{KK} in order to give a lower bound of $e^{-cN^2}$ on the persistence
of any GSP over $\Z$, provided that the spectral measure has a non-trivial absolutely continuous part. This gave rise to other interesting questions, stated in \cite{KK} and related to us also by M. Sodin \cite{Sodin}:
\begin{itemize}
\item Is there a GSP that achieves a persistence of the order of $e^{-cN^2}$?
\item  Is it possible for a GSP over $\R$ to have an even lower persistence?
\end{itemize}

In this paper we combine the spectral methods of \cite{FF} with tools from real and harmonic analysis in order to provide nearly complete
answers to all of these questions, in the case where the spectral measure has a non--trivial absolutely continuous component. While our methods do not employ~\cite{ABMO} directly, they are nonetheless inspired by the behavior of the sinc-kernel process.
Our results promote a point of view which regards persistence as a spectral property, governed by the interplay between the spectral behavior near zero and near infinity.

\subsection{Mathematical overview and discussion}\label{sec: overview}

Let $T\in\{\Z,\R\}$. A \emph{Gaussian process} on $T$
is a random function $f:T\to\R$ whose finite marginals, i.e. $(f(t_1),\dots, f(t_n))$ for any $t_1,\dots, t_n\in T$,
have multi-variate centered Gaussian distribution.
We say that $f$ is \emph{stationary} if its distribution is invariant under translations by elements of $T$. For an introduction to Gaussian processes see \cite{AT}.

The persistence probability of a Gaussian stationary process (GSP) $f$ on $[0,N]$ is defined by
\[
P_f(N):=\Pro\Big(f(t)>0, \: \forall t\in (0,N]\cap T\Big).
\]
Notice that we consider persistence above the mean (the zero level).
Our methods may be applied to study other constant levels, though we expect some qualitative differences in the results.

A GSP is determined uniquely by its covariance kernel
$ r(t)=\E [f(0)f(t)]$, $t\in T.$
Throughout the paper we implicitly assume Gaussian stationary processes to be almost-surely continuous and with a continuous covariance kernel.
Since $r$ is continuous and positive-definite, Bochner's theorem implies that there exists a
finite, symmetric, non-negative measure $\rho$ on $T^*$ (where $\R^*$ is identified with $\R$, and $\Z^*$ is identified with $[-\pi,\pi]$) such that
 \begin{equation}\label{eq: r rho}
r(t) = \widehat{\rho}(t) = \int_{T^*} e^{-i \lm t} \ d\rho(\lm).
\end{equation}
The measure $\rho$ is called \emph{the spectral measure} of the process $f$.
It is well known that any finite, symmetric, non-negative measure on $T^*$ corresponds to a unique GSP on $T$ (see Lemma~\ref{lem: basis} for a construction).
\bigskip

While our main results are presented in Section \ref{sec: results}, we
state here a simplified version which demonstrates our findings for particularly well-behaved spectral measures.
We write $A(N)\lesssim B(N)$ to denote that $A(N) \le C \, B(N)$ for some $C>0$ and all $N$, and $A(N)\asymp B(N)$ to denote that both $A(N)\lesssim B(N)$ and $B(N)\lesssim A(N)$.

\begin{thm}\label{Thm: main presentable}
Let $f$ be a GSP over $\R$ or $\Z$.
Suppose that its spectral measure is absolutely continuous with density $w(\lm)$ satisfying $\int |\lm|^\delta w(\lm)d\lm <\infty$ for some $\delta>0$, and $c_1|\lm|^\al \le w(\lm)\le c_2|\lm|^\al$ for all $\lm$ in a neighborhood of $0$ (and some $\al >-1$, $c_1,c_2>0$). Then , for large enough $N$:
\begin{align*}
\log P_f(N)
\begin{cases}
\asymp - N^{1+\al}\log N, & \al <0 \\
\asymp - N,  & \al=0 \\
 \lesssim - N \log N, & \al>0.
\end{cases}
\end{align*}
Moreover, if $w(\lm)$ vanishes on an interval containing $0$, then $\log P_f(N)\lesssim {-N^2}$.
In this case, if the process is over $\R$ and it satisfies in addition $w(\lm) \ge \lm^{-\eta}$ for some $\eta>0$ and all $|\lm|>1$, then $ \log P_f(N) \le -e^{CN}$.
\end{thm}

\smallskip

In Section \ref{sec: results} we provide results which are more general than Theorem \ref{Thm: main presentable}.
{The results are given by
\emph{the spectral mass near the origin} $\rho([0, \tfrac 1 N])$.
We do not require $\rho$ to have density; but rather a \emph{non-trivial absolutely continuous component} for the upper bounds, 
and \emph{no additional condition at all} for the lower bounds.
}

\medskip

For measures with a non--vanishing absolutely continuous component, our results prove the \emph{``spectral gap conjecture''} \cites{KK, Sodin}. This conjecture states that any process whose spectral measure vanishes on an interval around $0$ should have persistence smaller than $e^{-CN^{2}}$.
Prior to this paper there has not been any rigorous example of a GSP whose persistence decays faster than the order $e^{-C N \log N}$, although it was believed that the lower bound of $e^{-C N^2}$ given by
 Krishna-Krishnapur in~\cite{KK} for processes over $\Z$ should be attainable.
In a later work, joint with B.~Jaye and F.~Nazarov, we further developed the techniques of the present paper in order to establish the spectral gap conjecture in full generality. This result appears in a separate paper \cite{FFJNN}, which concentrates on the case of a spectral gap alone. 
 

Interestingly, our stochastic result regarding a spectral gap corresponds with the following analytic theorem of Eremenko-Novikov \cite{EN}:
The Fourier transform of a measure with a spectral gap has a positive asymototic density of zeroes on the real line.
Roughly speaking, both results reflect the idea that functions with a spectral gap have a strong tendency to oscillate.
Another result in this flavor was obtained recently by Borichev-Sodin-Weiss \cite{BSW}. They showed that any finite-valued stationary process on $\Z$ which has a spectral gap must be periodic, thus giving a probabilistic counterpart to a theorem by Helson (see references within \cite{BSW}).

\medskip

One may notice that no matching lower bound is given in Theorem \ref{Thm: main presentable} for the case $\al>0$. Over~$\Z$ we give such a matching lower bound (see Proposition \ref{cor: low exp} below), but over $\R$ this is impossible. This is due to an interesting phenomenon which happens only in continuous time:
\emph{When the spectral measure vanishes at $0$, then the heavier is the spectral tail at infinity -- the smaller is the persistence probability.}
This phenomena is reflected in the estimate $\log P_f(N)\le -e^{CN}$ appearing in Theorem~\ref{Thm: main presentable}, and a precise formulation of it appears in Proposition \ref{cor: small}.
A possible interpretation is that the heavy tail makes the process very rough, and this non-smoothness makes it even harder to stay positive (as opposed to a smooth process, for which positivity at a certain point makes it more likely for its whole neighborhood to be positive).
However, if the spectral measure is compactly supported, we believe matching lower bounds should hold (as is the case over $\Z$). This remains to be studied.

\medskip
As noted earlier, one novelty in our work is the ability to capture persistence behavior without requiring absolute summability or non-negativity of correlations.
For non-negatively correlated processes, that is, when $r(t)\ge 0$ for all $t\in T$,
the asymptotic behavior of $P_f(N)$ could be obtained
directly from $r(t)$ without using the spectral measure.
This was done by Dembo-Mukherjee, first in \cite{DM1} for the case $\al=0$ and later in \cite{DM2} for $\al<0$ (using the notation of Theorem~\ref{Thm: main presentable}).
Notice that, when $r(t)\ge 0$ the spectral measure at $\lm=0$ cannot vanish, so the case $\al>0$ is impossible.

\subsection{Outline of the paper}
The rest of the paper is organized as follows.

Section \ref{sec: results} contains {more precise} formulations of our results.
We present several propositions with explicit upper and lower bounds under spectral conditions (which imply, in particular, Theorem \ref{Thm: main presentable}). These propositions are all corollaries of Theorems \ref{thm: balance low} and \ref{thm: balance up}, which we formulate in Sections~\ref{sec: low} and~\ref{sec: up} respectively.
In Section \ref{sec: results} we also present Theorem \ref{thm: anal} which is an analytic tool in the flavor of the persistence results developed in this paper.
 Section \ref{sec: prelim} contains useful tools, such as spectral properties and decompositions of GSPs, ball and tail estimates, and well-known Gaussian inequalities.
Section \ref{sec: low} is dedicated to the proofs of lower bounds:
We develop a general inequality (Theorem \ref{thm: balance low}), from which we deduce Propositions \ref{cor: low exp} and \ref{cor: low}.
Section \ref{sec: up} is dedicated to the proofs of upper bounds:
We develop a general inequality (Theorem \ref{thm: balance up}) from which we deduce Propositions \ref{cor: up}, \ref{cor: up inf} and \ref{cor: small}.
Finally, Section \ref{sec: anal} contains the proof of the analytic (non-probabilistic) result, Theorem \ref{thm: anal}, which is used in Section \ref{sec: up}.

\section{Results}\label{sec: results}
In this section we provide a more precise presentation of our results.
Let $f$ be a GSP over $\R$ or $\Z$ with spectral measure $\rho$.
For $\delta\in \R$ denote the $\delta$-moment of $\rho$ by
\[
 m_\delta= m_\delta(\rho):=\int_{T^*} |\lm|^\delta \ d\rho(\lm).
 \]
Throughout, we assume that $\rho$ is normalized and has some finite positive moment, that is,
\begin{equation}\label{eq: basic cond}
\exists \delta>0: \, m_\delta<\infty, \quad \text{and} \quad m_0=\rho(T^*)=1.
\end{equation}
To capture the spectral behavior near zero we will employ both negative moments and the total measure on small intervals which we denote by
\[
\si^2_N:= \rho([0,\tfrac 1 N]), \quad \text{ for } N>0.
\]
These two ways to describe a measure near $0$ are related, e.g. by the following (see Observation~\ref{obs: IBP}):
\begin{equation}\label{eq: IBP items}
\begin{cases}
\text{If $\si_N^2 \le b N^{-(\g + \ep)}$ for some $b,\ep>0$ and all $N>0$, then $m_{-\g} < \infty$.} \\
\text{If $m_{-\g} < \infty$, then $\si_N^2\le b N^{-\g}$ for some $b>0$ and all $N>0$.}
\end{cases}
\end{equation}


The absolutely continuous component of $\rho$ is denoted by $\rho_{ac}$, and the notation
$\rho_{ac}\neq 0$ means that it is not trivial.
The support of $\rho_{ac}$ is denoted by $\sprt(\rho_{ac})$, and $|E|$ denotes the Lebesgue measure of a set $E$.
We reserve the letter $Z$ to denote a standard Gaussian random variable, i.e. $Z\sim\calN(0,1)$.

\subsection{Lower bounds}
Here we provide explicit
 lower bounds for the persistence probability.
{These bounds do not require the spectral measure to have a non-trivial absolutely continuous part,
 and depend only on the spectral mass near the origin.}
 In Section \ref{sec: low} we provide a more general lower bound (Theorem \ref{thm: balance low}) from which these {propositions} follow.
Throughout, we assume $f$ is a GSP with spectral measure $\rho$ satisfying the conditions in  \eqref{eq: basic cond}.
\begin{prop}[explicit lower bounds]\label{cor: low exp}
Assume that $m_\delta<\infty$ for some $\delta>0$.
Assume that $\si^2_N\ge b N^{-\g}$ holds on a subsequence of $\N$, where $b,\g>0$ are some constants. Then
there exists $C,N_0>0$ such that along this subsequence, for $N>N_0$,
\begin{align*}
\log P_f(N) \ge
\begin{cases}
-C N^{\g}\log N, & \g <1 \\
 -C N,  & \g=1 \\
  -C N \log N, & \g>1,\ T=\Z.
\end{cases}
\end{align*}
Here $N_0$ and $C$ depend only on $b,\g,\delta, m_\delta$.
\end{prop}

Note that in the case $\g>1$, which corresponds to vanishing spectrum at the origin, we give a lower bound only over $T=\Z$. As was discussed in Section \ref{sec: overview}, a similar lower bound is not true over $\R$ (see Proposition \ref{cor: small} below).
However, over $\Z$ we obtain additional estimates in the case of deeply vanishing spectrum at the origin.

\begin{prop}[lower bounds for vanishing spectrum]\label{cor: low}
Over $T=\Z$, assume that for a certain $N$ we have $N\si_N^2 <1$. Then for some universal $C>0$ we have
\[
\log P_f(N) \ge C N ( \log(N\si_N^2)-1).
\]

\end{prop}
In particular, if $\si_N^2 \ge e^{-AN^{\alpha}}$ for some
$A>0$ and $0<\alpha\leq 1$,
then $\log P_f(N) \ge -C' N^{1+\al}$.
 Note that Proposition \ref{cor: low} reproduces the cases $\g=1$ and $\g>1$ of Proposition \ref{cor: low exp} (over $\Z$).
However, it does not capture the lower bound of $e^{-c N^2}$ which holds for any spectral measure with density, as proven by
 Krishna-Krishnapur in \cite{KK}. 

\subsection{Upper bounds}
Here we state explicit upper bounds for the persistence probability.
In Section \ref{sec: up} we provide a more general upper bound (Theorem \ref{thm: balance up}) from which these propositions follow.
Throughout, we assume $f$ is a GSP with spectral measure $\rho$ satisfying the conditions in  \eqref{eq: basic cond}.
\begin{prop}[explicit upper bounds]\label{cor: up}
Assume that $\rho_{ac}\ne 0$ and 
{$\si_N^2 \le b N^{-\g}$ for all $N>0$ and some $b, \g>0$.}
Then there exist $N_0, C>0$ such that for all $N>N_0$:
\begin{align*}
\log P_f(N) \le
\begin{cases}
-C N^{\g}\log N, & \g <1 \\
 -C N,  & \g=1 \\
  -C N \log N, & \g>1.  
\end{cases}
\end{align*}
Here $N_0$ and $C$ depend on
 $b,\g, m_{-\g+\ep}, m_{-2k}, \nu, E$, where $\ep>0$ is arbitrary, $E$ is a set on which $d\rho_{ac} \ge \nu dx$ and $\g=2k+s$ with $k\in \N_0$ and $s<2$.
The dependence of $C$ on $E$ is linear in $|E|$.
\end{prop}

If the spectral measure has an infinite order zero at the origin, then the persistence becomes much smaller.
In particular, the estimate $e^{-cN^2}$ for the persistence of a process with a spectral gap is implied by the following proposition. In addition, this proposition provides an interpolation between this estimate and the results stated in Proposition  \ref{cor: up}.

\begin{prop}[upper bounds for vanishing spectrum]\label{cor: up inf}
Assume that $\rho_{ac}\ne 0$ and $m_{-4}<\infty$.
Then for large enough~$N$ we have
\[
\log P_f(N) \le -C N k (N) \log \left(\tfrac{cN}{k(N)}\right),
\]
where $C$ and $c$ are positive constants depending on $\rho$, and
$k=k(N)$ is any integer such that 
\[
1\le k\le \min(1,c)N, \quad k \, m_{-2k}^{1/k}\le  N.
\]

\end{prop}
In particular:
\begin{itemize}

\item If  $\rho\equiv 0$ on $[-\delta, \delta]$ for some $\delta>0$, then $m_{-k} < \delta^{-k}$ for all $k>0$ and $\log P_f(N) \lesssim -  N^2$.

\item  Let $A>0$. If $m_{-2k} < k^{A k}$ for all $k>0$, then $\log P_f(N) \lesssim -N^{1+\frac 1{1+ A}}\log N$.\\
(An example of this behavior is given by the spectral density $e^{-\frac 1{|\lm|^A}} \mathbf{1}_{[-1,1]}(\lm)$.)


\item If $m_{-2k_0}<\infty$ for some $k_0\ge 2$ but $m_{-2k}=\infty$ for $k>k_0$, then 
$\log P_f(N) \lesssim - N \log N$.\\ (this is implied also by the case $\g>1$ of Proposition \ref{cor: up}.)

\end{itemize}

The next result shows that over $\R$, if the spectral measure vanishes at zero and has a heavy enough tail at infinity, then the persistence probability is tiny.

\begin{prop}[tiny persistence]\label{cor: small}
Let $f$ be a GSP over $\R$, whose spectral measure $\rho$ has an absolutely continuous component with density $w(\lm)$.
Let {$\al>1$}. There exists a constant $C>0$ such that the following hold for large enough $N$.
\begin{enumerate}
\item If $\rho=0$ on $[-1,1]$ and $w(\lm) \ge \lm^{-\al}$ for $|\lm|>1$, then $\:\:\log P_f(N) \le -e^{CN}$.

\item If $m_{-4}(\rho)<\infty$ and $w(\lm) \ge \lm^{-\al}$ for $|\lm|>1$, then $\:\:\log  P_f(N) \le - C N^{1+
\frac 2 {\al}}\log N$.

\end{enumerate}
\end{prop}

Using our methods one can generate many more such examples, for instance, it is possible to get $\log P_f(N) \le -e^{C\sqrt N}$ with only first-order vanishing at $0$ (by using a tail of $\frac 1{\lm \log^2 \lm}$ on $[1,\infty)$).

\bigskip
We note that by combining Proposition \ref{cor: low exp} with Proposition \ref{cor: up} we obtain Theorem \ref{Thm: main presentable}. 
The `moreover' part of Theorem \ref{Thm: main presentable} follows from the first cases of Propositions \ref{cor: up inf} and \ref{cor: small}.

\subsection{Limitations of analysis near the origin}
Proposition~\ref{cor: small} indicates that for spectral measures which nullify near the origin,  the asymptotic behavior of $\log P_f(N)$ is influenced by the decay of the spectral measure near infinity. 
However, for compactly supported measures
the asymptotic behavior of $\log P_f(N)$ appears to be governed, up to a constant, by the spectral behavior near the origin. Nevertheless, even in this case, the leading constant itself cannot be determined by this local behavior, as it is affected by the entire measure. 
We demonstrate this through the following example, which is proved in Section~\ref{sec: up}.
\begin{prop}[leading constant]\label{cor: leading}
There exist GSPs $f_1$ and $f_2$, whose spectral measures are supported on $[-\pi,\pi]$ and coincide in a neighborhood of $0$, for which
\[
C_1<-\frac {1}{N} \log P_f(N)   < C_2 < -\frac {1}{N}\log P_{f_2}(N) < C_3.
\]
for some constants $C_1,C_2,C_3\in (0,\infty)$ and all sufficiently large $N$.
\end{prop}


\subsection{An analytic result}

Finally, we present an analytic result (in the flavor of persistence) which plays a role in proving our upper bounds.
It quantifies the fact that, 
if a function on an interval takes values in 
$[-L,L]$ and has a positive $k$-th derivative,
then the average of this $k$-th derivative cannot be too large with respect to $L$.
The statement holds true both on $T=\Z$ and on $T=\R$, where over $\Z$ the notion of derivative is the usual discrete one (see \eqref{eq: Delta} below)
and integrals are replaced by sums.

\begin{thm}[deterministic result]\label{thm: anal}
Let $T\in \{\Z, \R\}$ and $N>0$. Fix $k\in \N$ such that $k\le N$.
Suppose that $f:T\to \R$
is $k$-times differentiable and $f^{(k)}>0$ on $[-N,N]\subset T$.
Then
\[
\frac 1{N} \int_{-\tfrac{9}{20} N}^{\tfrac 9{20} N} f^{(k)} \le \left(\frac{c_0 k}{N}\right)^k \
\sup_{[-N,N]} |f|,
\]
where $c_0>0$ is a universal constant.
\end{thm}

Results of this type play a role in obtaining Remes-type inequalities in the context of approximation theory. 
Theorem~\ref{thm: anal} is a variant of a theorem due to Bernstein~\cite{Bern}, which implies that if $\inf_{[-N,N]} f^{(k)}> k!$ then
 $\sup_{[-N,N]} |f| \ge ({N}/ {c_0} )^k$ (see~\cite{Ganz}*{Sec. 2.1} for this result and a short discussion). It appears though, that Bernstein's proof has never appeared in English, and
as neither of the theorems implies the other, 
we provide a proof of our own in Section~\ref{sec: anal}. This proof also employs results from approximation theory and relies on the fact that 
the $k$-th degree Chebyshev polynomial is in some sense extremal for this inequality.

\section{Preliminaries}\label{sec: prelim}
This section contains tools which will be used throughout our proofs. These tools include properties implied by finite spectral moments, decompositions of a GSP, basic calculus for GSPs (that is, properties of its derivative and anti-derivative processes), ball and tail estimates, and some famous Gaussian inequalities. We begin with some notations and then sort our tools by topic.

\subsection{Notation}
Recall that we let $T$ be either $\R$ or $\Z$ and  correspondingly, we let $T^*$ be $\R$ or $[-\pi,\pi]$. Let $\rho$ be a positive, symmetric, and finite measure over $T^*$ and denote by
 $\mathcal{L}^2_\rho(T^*)$ the Hilbert space of functions $\{ \p: \ \int|\p|^2 d\rho<\infty\}$ with the inner product
$ \langle \p_1, \p_2 \rangle = \int_{T^*} \p_1 \overline{\p_2} d\rho$.
In case $\rho$ has density $\ind_E$ with $E \subseteq T^*$ being a compact, symmetric measurable set, we abbreviate this space by $\mathcal{L}^2_E$.

Unless stated otherwise, we always assume
$f:T\to\R$ to be a GSP with spectral measure $\rho$ and covariance function $r$. As before, for $\delta\in \R$ we denote the $\delta$-moment of $\rho$ by
 $m_\delta=\int_{T^*} |\lm|^\delta \ d\rho(\lm)$. Recall that we assume $m_\delta < \infty$ for some $\delta>0$.

A set of integers $\Lambda\subset \Z$ is said to have positive (central) density if
\[
D^-(\Lambda) := \liminf_{N\to\infty} \frac {\left|\Lambda \cap [-N,N] \right|}{2N}>0.
\]

The symbol $\overset d=$ indicates equality in distribution (between two random variables or processes on the same space).
The symbol $\oplus$ is used to sum two independent processes.

\subsection{Spectral moments}
Below are a few observations regarding spectral moments.
The first relates negative moments with
the spectral mass near $0$, as was stated in~\eqref{eq: IBP items}.

\begin{obs}\label{obs: IBP}
Let $\g, \ep >0$.
\begin{itemize}
\item[i.] If $\rho([0,\lm]) \le b\lm^{\g + \ep}$ for some $b>0$ and all $\lm>0$, then $m_{-\g}<\infty$.
\item[ii.] If $m_{-\g}<\infty$ then $\rho([0,\lm]) \le b \lm^{\g}$ with $b=\tfrac{1}{2}m_{-\g}$.
\end{itemize}
\end{obs}

\begin{proof}
For the first item, assume that $\rho([0,\lm]) \le b\lm^{\g + \ep}$ and use integration by parts:
\begin{align*}
\int_0^\infty \frac{d\rho(\lm)}{\lm^\g } &= -\lim_{\lm\to 0} \frac{\rho([0,\lm])}{\lm^\g} + \g \int_0^\infty \frac{\rho([0,\lm])}{\lm^{\g+1}} d\lm
\\ &\le 0 + \g\left(b \int_0^1 \lm^{-1+\ep} d\lm +\int_1^\infty \frac {d\lm}{\lm^{1+\g}}  \right) < \infty.
\end{align*}
For the second part, notice that $\int_0^\infty \frac{d\rho(\lm)}{\lm^\g} \ge \frac{1}{\tau^\g} \int_0^{\tau} d\rho(\lm)$ for any $\tau>0$.
\end{proof}

The second observation relates positive moments with the behavior of the covariance function near $0$.
\begin{obs}\label{obs: short dist}
If $m_\delta<\infty$ for some $\delta \in (0,2]$, then
\[
\forall t \in T: \quad r(t) \ge r(0) -C(\delta) m_\delta |t|^\delta,
\]
where $C(\delta)$ is a positive constant depending only on $\delta$. In particular, $C(2)=\tfrac 12$.
\end{obs}

\begin{proof}
By definition $r(t)$ is symmetric and we have:
\[
r(0)-r(t)=\int_\R\Big(1-\cos(\lm t) \Big) d\rho(\lm)\le C|t|^\delta \int_\R |\lm|^\delta d\rho(\lm),
\]
where $C =C(\delta)= \sup_{x\in\R} \frac {1-\cos(x)}{|x|^\delta}<\infty$ is finite when $\delta\in(0,2]$.
In particular, $C(2)=\tfrac 1 2$.
\end{proof}

We conclude with a property of negative moments.
\begin{clm}\label{clm: finite tau}
For any spectral measure $\rho$, there exists $\tau\in (0,\infty)$ such that
$m_{-2k+2} \le \tau \cdot m_{-2k}$ for all $k\in\N$.
\end{clm}

\begin{proof}
Let $M>1$ be such that $\int_0^M d\rho \ge \frac 1 2 \int_0^\infty d\rho$. This implies that
 $\int_0^M d\rho \ge \int_M^\infty d\rho$ and so, we have for every $j\in N$,
\[
\int_{0}^M \frac{d\rho(\lm)}{\lm^{j} } \geq \frac{1}{M^j} \int_0^M d\rho(\lm) \geq \frac{1}{M^j} \int_M^\infty d\rho(\lm)\geq \int_M^\infty \frac{d\rho(\lm)}{\lm^{j} }.
\]
It follows that  for every $j\in N$ we have
$\int_{0}^M \frac{d\rho(\lm)}{\lm^{j} }\geq \frac 1 2 \int_0^\infty \frac{d\rho(\lm)}{\lm^{j}} .$
We can therefore conclude that
\[
{\frac{1}{2}} m_{-2k}\geq \int_{0}^M \frac{d\rho(\lm)}{\lm^{2k} }\geq \frac{1}{M^2}\int_{0}^M \frac{d\rho(\lm)}{\lm^{2k-2} }\geq \frac{1}{2M^2}\int_0^\infty \frac{d\rho(\lm)}{\lm^{2k-2} }= \frac{1}{4M^2}\ m_{-2k+2}.
\]
\end{proof}

\subsection{Decomposition of a GSP}\label{sec: decomp}
\subsubsection{Spectral decomposition}

Assume that $\rho_{ac}\ne 0$.
This condition may be written more explicitly as follows: There exists a number $\nu>0$ and a bounded
set of positive measure $E\subset T^{*}$ such that
\begin{equation}\label{eq: AC cond}
d\rho = \nu \ind_E(\lm)d\lm + d\mu,  \quad \text{ where $\mu$ is a non-negative measure}.
\end{equation}
By rescaling $f$ we will assume that $E\subset [-\pi,\pi]$ (see Claim \ref{clm: E in pi}).
The next decomposition is an extension of~\cite{FF}*{Obs. 1}, where it was assumed that $E$ is an interval.

\begin{clm}\label{clm: Riesz}
Suppose that condition \eqref{eq: AC cond} holds with $E\subseteq [-\pi,\pi]$.
Then, there exist $\Lambda=\{\lm_n\}\subset \Z$ of positive density $A$, and a constant $B>0$, such that $f(\lm_n) \overset{d}{=} B Z_n \oplus g_n$, where $Z_n$ are i.i.d. standard Gaussian random variables and $g_n$ is a Gaussian process on $\Z$.

\noindent
Moreover, given $\ep>0$ one can have $A=(1-\ep)\frac{|E|}{2\pi}$, $B= \sqrt{c(\ep) \nu |E|}$, where $c(\ep)>0$ is a constant depending only on $\ep$.
\end{clm}

The proof of Claim~\ref{clm: Riesz} is based on the following result by Bourgain and Tzafriri, which is a consequence of their celebrated ``Restricted Invertibility Theorem'' \cite{BT}. The `moreover' part is due to Vershynin \cite{Ver}*{Thm. 1.5}.

\begin{thma}[Bourgain, Tzafriri, Vershynin]\label{thm: ver}
Let $E\subseteq [-\pi,\pi]$ be a set of positive Lebesgue measure.
Then, there exist $\Lambda=\{\lm_n\}\subset \Z$ and constants $A, D>0$, such that:
\begin{enumerate}[{\rm (i)}]
\item\label{itm: I} $\forall \{a_n\}\in l^2(\Z): \quad D\sum|a_n|^2 \le \norm{\sum a_n e^{-i\lm_n x}}_{L^2(E)}^2 \le  \sum |a_n|^2$.
\item\label{itm: II} $\liminf_{N\to\infty} \frac {\left|\Lambda \cap [-N,N] \right|}{2N}>A$.
\end{enumerate}
Moreover, given $\ep > 0$ one can have $A=(1-\ep)\frac{|E|}{2\pi}$ and ${D=c(\ep)|E|}$, where $c(\ep)>0$ is a constant depending only on $\ep$.
\end{thma}

\begin{proof}[Proof of Claim~\ref{clm: Riesz}]
 Let $\ep>0$, and let $\Lambda=\{\lambda_n\}\subset \Z$ be the sequence whose existence is guaranteed by Theorem \ref{thm: ver}.
We have:
\[
\norm{ \sum_n a_n e^{-i\lm_n x}}^{2}_{L^2_{\rho}} = \sum_{n,m} a_n a_m\int e^{-i(\lm_n-\lm_m)x} d\rho = \sum_{n,m} a_n a_m r(\lm_n-\lm_m) =
a^T \Sigma a,
\]
where $\Sigma = ({r}(\lm_n-\lm_m) )_{n,m\in \Z}$ is the infinite covariance matrix of the Gaussian process $( f(\lm_n) )_{n\in\Z}$.
We note that this matrix is symmetric, as $r$ is symmetric.
By item \eqref{itm: I} of Theorem \ref{thm: ver}, and  condition \eqref{eq: AC cond},  we have for all $\{a_n\}\in l^2(\Z)$,
\[
D\nu\sum|a_n|^2 \le \nu\norm{\sum a_n e^{-i\lm_n x}}_{L^2(E)}^2 \le \norm{\sum a_n e^{-i\lm_n x}}_{L^2_{\rho}}^2 =
a^T \Sigma a.
\]
It follows that $\Sigma -\nu D I$ defines a positive-definite operator on
$\ell^2(\Z)$ (here $I(n,m)=\ind\{n=m\}$ is the identity).
Therefore, it is the covariance of some Gaussian process $g:\Z\to\R$ (see e.g. \cite{Lif}*{Sec. 4, Thm. 2}).
We obtain that
\[
f(\lm_n) \overset{d}= \sqrt{\nu D} Z_n \oplus g_n, \quad Z_n \sim \calN(0,1) \text{  i.i.d.}
\]
This establishes the result with $B=\sqrt{\nu D}=\sqrt{c(\ep) \nu |E|} $.
\end{proof}

The decomposition in Claim~\ref{clm: Riesz} will be useful to us together with the following two claims, the first of which appeared in \cite{FF}*{Prop. 3.1}.

\begin{clm}\label{clm: iid}
Let $(Z_j)_{j\in\Z}$ be a sequence of i.i.d. centered Gaussian random variables. Let $q, b_1,\dots,b_N \in\R$ be numbers such that $\frac{1}{N}\sum_{j=1}^N b_j \le q$. Then:
\begin{align*}
\Pro \left( Z_j +b_j{\geq} 0, \: 1\le j\le N \right)  \leq  \Pro(Z_1 \leq q)^N.
\end{align*}
\end{clm}

\begin{proof}
Write
$\Phi(b) :=\Pro(Z_1\le  b) = \Pro(Z_1 \geq -b)$.
One may check that
$x\mapsto \log \Phi(x)$ is monotone and concave (for $x>0$ it is straightforward, for $x\le 0$ one should use
the tail estimate in Lemma~\ref{lem: tail}(a)). Thus we have:
\begin{align*}
\log \Pro\left(Z_j+b_j \geq 0,\: 1\le j\le N\right)
= \sum_{j=1}^N \log \Phi(b_j)
\le N \cdot\log
\Phi\left({\frac{1}{N}\sum_{j=1}^N b_j }\right) \leq N\cdot \log \Phi(q).
\end{align*}

\end{proof}

\begin{clm}\label{average}
Let $f:\Z/N\Z\rightarrow \R$ be a function satisfying $\frac{1}{N}\sum_{n=0}^{N-1}f(n)\leq L$, for some $L\in\R, N\in\N$. Then, for every $S\subseteq \Z/N\Z$ there exists $\tau\in\Z/N\Z$, such that \[\frac{1}{|S|}\sum_{n\in S}f(n+\tau)\leq L.\]
\end{clm}
\begin{proof}
Let $g:\Z/N\Z\rightarrow \R$ be the function defined by $g(\tau)=\frac{1}{|S|}\sum_{n\in S}f(n+\tau)$. Then $\frac{1}{N}\sum_{\tau=0}^{N-1}g(\tau)\leq L$, which implies that
there exists $\tau\in\Z/N\Z$ such that $g(\tau)\leq L$.
\end{proof}

\subsubsection{Hilbert decomposition}\label{fourdec}
We turn to a different type of decomposition.
The next proposition gives a classical \emph{basis representation} of {GSP's}.

\begin{lem}\label{lem: basis}
Let $\rho$ be a symmetric, non-negative measure on $T^*$ with a finite positive moment, and let $\p_n$ be an orthonormal basis in $\mathcal{L}^2_\rho(T^*)$ which satisfies, for every $n\in\N$, $\p_n(-\lambda)=\overline{\p_n(\lambda)}$.
Denote $\psi_n(t) = \int_\R e^{-i\lm t}\p_n(\lm) d\rho(\lm)$. Then
\[
f(t) = \sum_n \zeta_n \psi_n(t), \quad \zeta_n \sim \calN(0,1) \text{  i.i.d.}
\]
is a continuous GSP over $T$ with spectral measure $\rho$.
\end{lem}
 We note that every space  $\mathcal{L}^2_\rho(T^*)$, with a measure $\rho$ satisfying the requirements above, admits such an orthonormal basis. In this case, the condition on the elements of the basis, $\p_n(-\lambda)=\overline{\p_n(\lambda)}$, implies that the functions $ \psi_n(t)$ are real.

\begin{proof}
Standard arguments (see \cite{Kah}*{Chapter 3, Thm. 2} or \cite{GAFbook}*{Lemma 2.2.3}) yield that the series defining $f$ converges almost surely to a Gaussian function, with covariance
\[
K(t,s)= \E [f(t)f(s)] = \sum_n \psi_n(t)\psi_{n}(s).
\]

Denote $e_{t}(\lm)=e^{i\lm t}$. Since $\{\p_n\}$ is an orthonormal basis in $\mathcal{L}^2_\rho(T^*)$, we have

\[
\sum_n \psi_n(t)\psi_{n}(s) =\sum_n\langle\p_n,e_t\rangle_{\mathcal{L}^2_\rho(T^*)}\langle \p_{n}, e_s\rangle_{\mathcal{L}^2_\rho(T^*)}=
\langle e_t,e_s \rangle_{\mathcal{L}^2_\rho(T^*)} =\widehat{\rho}(t-s).
\]
Thus $f$ is stationary with spectral measure $\rho$. Almost sure continuity of $f$ follows from the moment condition on
$\rho$ (in fact, it follows from the weaker condition
$ \int \log^{1+\ep}(1+|\lm|) d\rho(\lm)<\infty$ for some $\ep>0$, see \cite{AT}*{Ch. 1.4.1}).
\end{proof}

From Lemma \ref{lem: basis} we deduce the following claim.
\begin{clm}\label{clm: one}
Let $\p \in \mathcal{L}^2_\rho$ be a real symmetric function such that $\|\p\|_{\mathcal{L}^2_\rho}=1$. Write $\psi (t) = \int_\R e^{-i\lm t}\p(\lm) d\rho(\lm)$, then we have the decomposition
\[
f(t) \overset d{=} \zeta \cdot \psi(t) \oplus g,
 \]
 where $\zeta\sim \calN (0,1)$ and $g$ is a Gaussian process which is independent of $\zeta$.
\end{clm}

\begin{proof}
Such a function $\p$ can be completed into a basis of $\mathcal{L}^2_\rho$, which satisfies the conditions of Lemma \ref{lem: basis}. The result immediately follows.
\end{proof}

\subsection{Calculus of GSP's}
Next we discuss the relationship between the
spectral measure and differentiation or integration.
The derivative of a function $f:T\to\R$, where $T\in\{\Z,\R\}$, is defined to be
\begin{equation*}
\displaystyle f'(t) = \begin{cases}
\lim_{\delta\to 0} \frac{f(t+\delta)-f(t)}{\delta}, & T=\R\\
f(t+1)-f(t), & T=\Z.
\end{cases}
\end{equation*}

\begin{obs}[derivative]\label{obs: rate}
Suppose that $m_\delta<\infty$ for some $\delta>2$.
Then $f$ is a.s. continuously differentiable, and
$f'$ is a GSP with spectral measure $\mu$ defined by:
\begin{equation}
d\mu(\lm)=
\begin{cases}
\lm^2 d\rho(\lm), & T=\R \\
2(1-\cos \lm )d\rho(\lm) & T=\Z.
\end{cases}
\end{equation}
\end{obs}

\begin{proof}
In the case $T=\R$, the fact that $f'$ exists a.s. and is a stationary continuous Gaussian process follows from {the} moment condition
(see \cite{AT}*{Ch. 1.4.1}).
Differentiating the relation
$\E[f(t) f(s)] = \int_\R e^{-i\lm (t-s)}d\rho(\lm)$
once by $t$ and once by $s$ yields
\[
\widehat{\mu}(t-s)=\E[f'(t)f'(s)] = \int_\R e^{-i\lm(t-s)}\lm^2 d\rho(\lm).
\]

In the case $T=\Z$, differentiability is immediate and we compute:
\begin{align*}
\widehat{\mu}(m-n)&=\E[ f'(m) f'(n) ] =\ \E[ (f(m+1)-f(m))(f(n+1)-f(n))] \\
&=2r(m-n)- r(m-n+1)- r(m-n-1) \\
&= \int (2- e^{-i\lm}- e^{i\lm}) e^{-i\lm(m-n)} d\rho(\lm)=\int e^{-i\lm(m-n)}2(1-\cos \lm) d\rho(\lm).
\end{align*}
\end{proof}

\begin{obs}[stationary anti-derivative]\label{obs: integ}
Suppose that $m_{-2}<\infty$ and $m_\delta<\infty$ for some $\delta>0$.
Then there exists a GSP $F:T\to\R$ such that $F' \overset{d}{=} f $.
\end{obs}

\begin{proof}
Let $\mu$ be the measure defined by

\begin{equation}
d\mu(\lm)=
\begin{cases}
\frac{d\rho(\lm)}{\lm^2}, & T=\R \\
\frac{d\rho(\lm)}{2(1-\cos \lm )} & T=\Z.
\end{cases}
\end{equation}
By the premise $\mu$ is a finite, non-negative, symmetric measure, and therefore defines a GSP which we denote by $F$.
Moreover, in the case $T=\R$ we have $m_{2+\delta}(\mu) = m_\delta(\rho)<\infty$. Thus by Observation~\ref{obs: rate} it follows that $F$ is a.s. continuously differentiable and $F' \overset{d}{=} f$.
\end{proof}

Observation \ref{obs: integ} asserts that if $m_{-2}<\infty$ then the anti--derivative process is stationary, and in particular its variance is uniformly bounded. The next lemma, which is a generalization of Proposition 3.2 in \cite{FF}, provides estimates for the variance of the anti-derivative of a GSP even when the latter is not stationary.
We formulate and prove it in continuous time, noting that the discrete analogue follows by simple modifications.

\begin{lem}[general anti-derivative]\label{lem: average}
Let $b\ge 0$ and $\g\in[0,2)$.
Suppose that $\rho([0,\lm])\le b\lm^\g$ for all $\lm>0$.
Then for all $N>0$:
\[
\var\left( \int_0^N f(t) dt \right)\le
b \, C(\g) \, \var(f(0)) \cdot
N^{2-\g} ,
\]
where $C(\g) = \frac {16}{2-\g}$.
\end{lem}

\begin{proof}
Without loss of generality assume $\var ( f(0) )=1$.
 We calculate the variance:
\begin{align*}
 \var \left(\int_0^N f(t) dt \right)&=\E\left[ \left(\int_0^N f(t) dt\right)^2 \right]
  = \iint_{[0,N]^2} \E (f(t) f(s)) dt\ ds\\
  &=\int_0^N \int_0^N \widehat{\rho}(t-s) dt\ ds
  =N \int_{|t|<N}\left(1-\frac {|t|}{N}\right) \widehat{\rho}(t)dt.
\end{align*}
The change in order of integration and expectation follows from Fubini's theorem.
The inverse Fourier transform of ${N} (1-\frac {|t|}{N})\ind_{[-N,N]}(t)$ is given by $N^2\ \sinc^2(\tfrac{N}{2}\lm)$
where $\sinc(x) = \tfrac{\sin x}{x}$ (the definition of Fourier transform is given in \eqref{eq: r rho}). We use the estimate
\[
\sinc^2( N \lm/2) \le  \begin{cases}
1, & |\lm|\le \tfrac 1 N, \\
4 (N\lm)^{-2},  & |\lm|> \tfrac 1 N,
\end{cases}
\]
combined with Parseval's formula \cite{Katz}*{Sec 2.2}, to get:
\begin{align*}
\var \left( \int_0^N f(t) dt \right)
&= N^2 \int_\R \sinc^2( N\lm/2) d\rho(\lm)
\\& \le  2N^2 \rho([0,\tfrac 1 N]) + 8 \int_{1/N}^\infty  \frac {d\rho(\lm)}{\lm^2}
\\& = 2N^2\rho([0,\tfrac 1 N])+8  \left( 2\int_{1/N}^\infty \frac{\rho([0,\lm])}{\lm^3}d\lm - N^2\rho([0,\tfrac 1 N])  \right)
\\&
\le 16 b\int_{1/N}^{\infty}\lm^{\g -3} d\lm
\le \frac{16 b }{2-\g} N^{2-\g}.
\end{align*}


\end{proof}

Lastly we need estimates on the supremum of a GSP and its anti-derivative. We achieve this using Dudley's metric entropy bound \cite{AT}*{Thm. 1.3.3}, which reads as follows.
For a Gaussian process $H$ on $I$, we define a canonical semi-metric by
$
d_H(a,b) :=\sqrt{ \E ( H(a) - H(b) )^2 }.
$
For any $\ep>0$, let $N(\ep)$ be the minimal number of $d_H$-balls of radius $\ep$ which cover $I$. Then Dudley's bound states that there exists a universal constant $K>0$ such that
\begin{equation}\label{eq: dudley}
\E \sup_I H \le K \int_0^{\text{diam(I)}} \sqrt{\log N(\ep)} d\ep,
\end{equation}
where $\text{diam(I)}$ is the diameter of $I$ under $d_H$.
The following lemmas are applications of this bound.
  
  \begin{lem}\label{lem: sup stat}
Let $f$ be a GSP for which $m_0 = \int d\rho(\lm)$ and $m_2=\int \lm^2 d\rho(\lm)<\infty$. Denote $a=\sqrt{\frac{m_2}{4m_0}}$. Then there is a universal constant $K>0$ such that for all $N>1$ we have
\[
\E \left(\sup_{ [0,N]} f \right) \le K \sqrt{m_0 \cdot\max\{\log (a N ),1\}},
\]
\end{lem}

\begin{proof}
By stationarity we have $d_f(x,y) =  \sqrt{2(r(0) - r(x-y))}$, from which we deduce $\text{diam}([0,N])\le \sqrt{4m_0}$. Moreover, by Observation~\ref{obs: short dist},
$d_f(x,y)  \le \sqrt{m_2} |x-y|$, which implies that $N(\ep) \le \max(1, \sqrt{m_2} \frac{N}{ \ep})$ and $\text{diam}([0,N])\le \sqrt{m_2} N$.
We consider two cases. If $2\sqrt{4 m_0}\leq \sqrt{m_2}N$, then by Dudley's bound
\[
\E \sup_{[0,N]} f \le K\int_0^{\sqrt{4m_0}} \sqrt{\log  \left( \frac{\sqrt{m_2} N}{ \ep } \right) }d\ep
= K\sqrt{m_2}  N \int_{\sqrt{\tfrac{m_2}{4m_0}} N}^{\infty} \frac{\sqrt{\log u}}{u^{2}} du.
\]
Note that for {$A>1$} one has:
\[
\int_A^\infty \frac{\sqrt {\log u}}{u^{2}} du
= \frac{\sqrt{\log A}}{A} + \frac 1 2
\int_{A}^\infty \frac{1 }{u^2\sqrt{\log u}} du
 \le \frac{\sqrt{\log A}}{A}\left(1+ \frac 1 {2\log A}\right),
\]
{which implies that under these conditions} $\E \sup_{[0,N]} f \le {\widetilde{K}} \sqrt{m_0\log (a N)}$ where $a=\sqrt{\frac{m_2}{4 m_0}}$ and ${\widetilde{K}}$ is some universal constant.
On the other hand, if $2\sqrt{4m_0}\geq \sqrt{m_2}N$ Dudley's bound gives
\[
\E \sup_{[0,N]} f \le K\int_0^{\sqrt{m_2}N} \sqrt{\log  \left( \frac{\sqrt{m_2} N}{ \ep } \right) }d\ep
= K\sqrt{m_2}  N \int_{1}^{\infty} \frac{\sqrt{\log u}}{u^{2}} du
\le {\widetilde{K}} \sqrt{m_0},
\]
so the desired bound holds.
\end{proof}

\begin{lem}\label{lem: sup H}
Let $f$ be a GSP such that $\rho([0,\lm]) \le b\lm^\g$ for some $0\le \g < 2$ and all $\lm>0$. Then, for all $N>1$:
\[
\E \sup_{x\in [0,N]} \left(\int_0^x f(t) dt  \right) \le c(\g)\sqrt{b m_0}\, N^{1-\frac{\g}{2}},
\]
where $c(\g)>0$ is a constant depending only on $\g$ and $m_0=\int d\rho$.
\end{lem}

\begin{proof}
Denote $H(x) = \int_{0}^x f(t) dt$. Then $H$ is a Gaussian process, whose canonical semi-metric may be bounded by Lemma \ref{lem: average}:
\[
d_H(x,y) = \sqrt{\var \left(\int_x^y f(t) dt \right)} = \sqrt{\var \left(\int_0^{y-x}f(t) dt\right) } \le C \cdot|y-x|^{1-\frac {\g}{2}},
\]
where $C= c_0(\g) \sqrt{b m_0}$.
 From now on $c_j$, $j\in \N$, will denote constants which depend only on~$\g$.
Denoting $\al:=1-\frac{\g}{2}$,
we see that $\text{diam}([0,N]) \le C N^{\al}$, and $N(\ep) \le
\max\left(1,
\frac N {(\ep/C)^{1/\al}} \right)$ (by taking balls of Eucleadian radius $(\ep/C)^{1/\al}$).
By Dudley's bound \eqref{eq: dudley}, for $0\le \g < 2$ we have:
\begin{align*}
\E \sup_{x\in [0,N]} H(x)
& \le K \int_0^{C N^{\al}}
\sqrt{\log \left(\frac {C^{1/\al} N}{ \ep^{1/\al}} \right)}d\ep
{=\frac{K}{\sqrt{\al}}\int_0^{CN^\al} \sqrt{\log \left( \frac{CN^\al}{\ep} \right)} d\ep }\\
& = c_1(\g) C N^{\al} \int_{1}^\infty \frac{\sqrt{\log u}}{u^{2}} du
= c_2(\g) \sqrt{b m_0}  N^{1-\tfrac \g 2}.
\end{align*}


\end{proof}

\subsection{Ball and tail estimates}

The terms ``ball'' and ``tail'' events refer to a stochastic process remaining inside or outside a ball, respectively.
These have been immensely studied, see e.g. \cite{Li-Shao}.
The following bounds, which will be repeatedly used, are ball and tail estimates for the one-dimensional Gaussian variable $Z\sim \calN(0,1)$.

\begin{lem}\label{lem: tail}

 For all $x>0$:
\begin{enumerate}[{\rm (a)}]
\item $\frac 1{\sqrt{2\pi}}\left(\frac 1 x - \frac 1{x^3}\right) e^{-x^2/2} \le  \Pro(Z>x) \le \frac 1{\sqrt{2\pi}}\frac 1 x  e^{-x^2/2}$.\\
In particular, for $x\ge 2:\quad e^{-x^2} \le \Pro(Z> x)\le e^{-x^2/2}.$
\item $\sqrt{\frac{2}{\pi}} x e^{-x^2/2} \le \Pro(|Z| \le  x ) \le x.$ \\
In particular, for $0<x\le 1: \quad \frac 1 4 x \le  \mathbb{P}(|Z|\leq x )\le x$.
\end{enumerate}
\end{lem}
We omit the proof, as its first part is a standard bound on the Gaussian tail (see~\cite{AT}*{eq. (1.2.2)}), while the second part follows from a straightforward integral estimate.

The estimates in Lemma \ref{lem: tail} imply the following comparison of tail probabilities.
\begin{clm}\label{clm: tails comp}
For any $\delta>0$ there exists $\theta>0$ such that $\Pro\big(Z\le x\big) \le \Pro\big(|Z|\le \theta x\big)$ for all $x>\delta$.
\end{clm}

\begin{proof}
We first note that the inequality in the statement can be rewritten as $\Pro\big(Z> x\big)\ge \Pro(|Z|>\theta x)$. Let $\delta>0$. There exists $\theta_1=\theta_1(\delta)$ such that $ \Pro(Z>2) \ge  \Pro(|Z|>\theta_1\delta)$. Set $\theta=\max\{2,\theta_1\}$.
To show that the inequality above holds for all $x\geq \delta$ we consider two cases. First, assume that $x\geq 2$. Then,
by part (a) of Lemma \ref{lem: tail} we have
\[
\Pro(Z>x) \ge  e^{-x^2} \ge 2 e^{-(2 x)^2/2} \ge  \Pro(|Z|>2 x)\ge  \Pro(|Z|>\theta x).
\]
Next, consider the case where $\delta\leq x\leq 2$. Then,
\[
\Pro(Z>x) \ge  \Pro(Z>2) \ge  \Pro(|Z|>\theta_1\delta) \ge  \Pro(|Z|>\theta x).
\]
\end{proof}

\medskip

Now we turn to bound the ball probability of a Gaussian process.
For discrete time this is given by the Khatri-Sidak inequality \cite{ineq}*{Ch. 2.4}, which is a particular case of the recently proved Gaussian correlation inequality \cites{LM, Roy}.
\begin{lem}[Khatri-Sidak]\label{lem: KS}
 for any $\ell>0$ and any centered Gaussian vector $Z$ one has:
\begin{equation*}
\Pro(|Z_j|\le \ell, \quad j=1,\dots,N) \ge \prod_{j=1}^N\Pro(|Z_j|\le \ell).
\end{equation*}
\end{lem}
The next lemma extends this inequality to continuous time, provided that $\ell$ is large enough.
We use the standard ``chaining method'', which is nicely presented in \cite{Li-Shao}.

\begin{lem}[Large ball]\label{lem: prod}
Let $h$ be a GSP over $\R$ which satisfies $m_\delta<\infty$ with a given $\delta>0$. Then there exist $\ell_0>0$ and $c\ge 1$ such that
for all $\ell>\ell_0$ and all $N\in\N$ the following holds:
\[
\Pro\left(|h(t)|\le \ell, \;\; \forall t\in[0,N]\right) \ge \Pro\left( c |h(0)|\le \ell \right)^N.
\]
The constants $c$ and $\ell_0$ depend only on $\delta$, $m_\delta$ and $m_0$.
\end{lem}

We remark that the dependence of the constant $c$ on properties of the spectral measure is essential, as can be seen by 
scaling arguments. This is a major difference between continuous-time (i.e. Lemma \ref{lem: prod}) and discrete-time (i.e. Lemma \ref{lem: KS}).

\begin{proof}

Assume, without loss of generality, that $m_0=\var h(0) =1$. Fix a number $c\geq 1$ (to be specified later) and let $\ell_0\in \N$ be large enough such that, in particular,
\begin{equation}\label{lnot}
e^{-\left(\frac{\ell_0}{c}\right)^2}\leq \tfrac{1}{12}.
\end{equation}
By part (a) of Lemma \ref{lem: tail}, for any $\ell\ge \ell_0$ we have:
\begin{equation}\label{eq: single}
\Pro\left( |h(0)|\le \frac {\ell}{c} \right)^N \le (1- 2e^{-\frac{\ell^2}{c^2}})^N \le \exp(- 2N e^{- \frac{\ell^2}{c^2} }),
\end{equation}
where in the right inequality we used the fact that $\log(1-x)\leq -x$ for all $0<x<1$.
Let $\al=1+\sum_{k=1}^\infty \frac 1 {k^2}$. For $k\in\N$, define the event
\[
A_k ={ \bigcap_{j=1}^{N2^k} } \left\{ \big|h(j 2^{-k})-h ( (j-1)2^{-k}) \big|  \leq\frac {\ell}{\al k^2} \right\},
\]
while $A_0 = \bigcap_{j=0}^N \big\{|h(j)|\le \frac \ell \al \big\}.$
Since every real number equals $n + \sum_{k=1}^\infty \ep_k 2^{-k}$ for some $n\in\Z$ and $\ep_k\in\{0,1\}$, the almost-sure continuity of $h$ implies that,
\[
\bigcap_{k\ge 0} A_k \subset \{|h(t)|\le \ell, \;\; \forall t\in[0,N]\}.
\]
 Therefore,
\begin{align*}
\Pro\Big( |h(t)|\le \ell, \;\; \forall t\in[0,N]\Big)& \ge \Pro\Big( \bigcap_{k\ge 0} A_k \Big) =\lim_{K\to\infty} \Pro\Big(\bigcap_{0\le k\le K} A_k\Big).
\end{align*}
Now, by Observation \ref{obs: short dist} we have $\text{var}\big(h(2^{-k})-h(0)\big) =2(1-r(2^{-k}))\leq \beta 2^{-d k}$, where $d=\min(\delta,2)$ and $\beta=\beta(d,m_d)$ is a given constant.
By Lemma \ref{lem: KS}, stationarity and Lemma \ref{lem: tail}(a), we have:
\begin{align*}
\Pro(\cap_{0\le k\le K} A_k) &\ge \Pro\left(|h(0)| \leq \frac{\ell}{\al} \right)^{N} \prod_{1\le k\le K} \Pro\left(|h(2^{-k})-h(0)|  \leq \frac{\ell}{\al k^2} \right)^{N2^k}
\\ & \ge {\left(1- 2e^{-\frac{\ell^2}{2\al^2}} \right)^{N}}\prod_{1\le k\le K} \left(1- 2e^{-\frac{\ell^2 2^{d k}}{ 2\beta \al^2 k^4}} \right)^{N2^k}.
\end{align*}
Since $-2x\leq \log(1-x)$ for all $0<x<\frac{1}{2}$, we find that if $\ell_0$ is large enough then all $\ell \geq \ell_0$ satisfy
\[
 \Pro\Big( |h(t)|\le \ell, \;\; \forall t\in[0,N]\Big)\ge \exp\left(-4N  \left( e^{-\frac {\ell^2}{2\al^2}} +\sum_{k=1}^K 2^{k} e^{-\frac{\ell^2 2^{d k}}
{2\beta \al^2 k^4}} \right) \right).
\]
In light of~\eqref{eq: single}, we need only check that there is a choice of $\ell_0$ and $c$ so that for any $\ell>\ell_0$ one has:
\begin{equation*}
\exp\left(-{4}N  \left( e^{-\frac {\ell^2}{2\al^2}} +\sum_{k=1}^\infty 2^{k} e^{-\frac{ 2^{d k}}
{2\beta \al^2 k^4} \ell^2} \right) \right) {\ge} \exp(-{2}Ne^{-\frac {\ell^2}{c^2}} ),
\end{equation*}
which is equivalent to
\[
\sum_{k=1}^\infty 2^{k}  e^{-\frac{2^{d k}}{2\beta \al^2 k^4} \ell^2 } \le \frac 1{2}e^{-\left(\frac{\ell}{c}\right)^2 } - e^{-\frac 1 2\left(\frac {\ell}{\al}\right)^2}.
\]
{If $c^2\ge 4 \al^2$, it is enough to show that}
\[
{\sum_{k=1}^\infty 2^{k}  e^{-\frac{2^{d k}}{2\beta \al^2 k^4} \ell^2 }  \le \frac 1{2}e^{-\left(\frac{\ell}{c}\right)^2 } - e^{-  2\left(\frac {\ell}{c}\right)^2}},
\]
which by applying (\ref{lnot}) reduces to
\[
\sum_{k=1}^\infty 2^{k}  e^{-\frac{2^{d k}}{2\beta \al^2 k^4} \ell^2 }  \le \frac 1{3}e^{-\frac{\ell^2}{c^2} } .
\]
Setting $\ds \frac 2{c^2} = \min \left(\min_{k\in\N} \frac{2^{d k}}{2\beta \al^2 k^5}, \frac 1{2\al^2}\right)$ and using the inequality $\sum_{k=1}^\infty q^k \le 2q$ for $q<\frac 1 2$, we have:
\[
\sum_{k=1}^\infty 2^{k}  e^{-\frac{2^{d k}}{2\beta \al^2 k^4} \ell^2 }  \le \sum_{k=1}^\infty (2e^{-2\ell^2/c^2})^k \le 4 e^{-2\ell^2/c^2} \le \frac 1{3}e^{-\frac{\ell^2}{c^2} },
\]
where in each of the last two inequalities we used the estimate (\ref{lnot}). This completes the proof.
\end{proof}

\subsection{Two famous Gaussian inequalities}
We end with two famous Gaussian bounds. The first is a comparison between ball probabilities due to Anderson \cite{ineq}*{Ch. 2.3}.

\begin{lem}[Anderson]\label{lem: anderson}
Let $X, Y$ be two independent, centered Gaussian processes on $I$. Then for any $\ell>0$,
\[
\Pro\Big(\sup_{I} |X \oplus Y| \le \ell\Big) \le \Pro\Big(\sup_{I} |X|\le \ell\Big).
\]
\end{lem}

The second lemma is due independently to Borell and Tsirelson-Ibragimov-Sudakov, see \cite{AT}*{Thm. 2.1.1}.
\begin{lem}[Borell-TIS]\label{lem: BTIS}
Let $X$ be a centered Gaussian process on $I$ which is almost surely bounded. Then for all $u>0$ we have:
\[
\Pro\left(\sup_I X  - \E \sup_I X >u \right) \le \exp\left(-\frac{u^2}{2\si_I}\right),
\]
where $\si_I = \sup_{t\in I} \var X(t)$.
\end{lem}

\section{Lower bounds}\label{sec: low}

{The main result of this section is the following general inequality. }

\begin{thm}[general lower bound]\label{thm: balance low}
Let $f$ be a GSP with spectral measure $\rho$ obeying \eqref{eq: basic cond}, and let $\delta>0$ be such that
$m_\delta=m_\delta(\rho)<\infty$. Then there exist
\[
\beta=
\begin{cases}
\beta(\delta,m_\delta), & T=\R \\
2\sqrt{2}, & T=\Z
\end{cases}
\quad \text{and} \quad
\ell_0 = \begin{cases}
\ell_0(\delta,m_\delta), & T=\R \\
0, & T=\Z
\end{cases}
\]
such that, for all $\ell>\ell_0$ and $N>0$, we have:
\[
P_f(N) \ge \Pro\big( \si_N Z > \ell\big) \cdot \Pro\big( \beta|Z| <  \ell\big)^{ N}.
\]
\end{thm}

Theorem \ref{thm: balance low} gives a recipe for estimating $P_f(N)$ from below: given $N$ and $\rho$, one should choose a level $\ell=\ell(N,\rho)$ so that
the factors in the theorem's estimate
are of the same order. This recipe is used to derive Propositions \ref{cor: low exp} and \ref{cor: low}. We start by proving Theorem \ref{thm: balance low}.

\subsection
{General lower bound: proof of Theorem \ref{thm: balance low} }

\begin{proof}
We use the Hilbert decomposition discussed in Subsection \ref{fourdec}. Fix $N>0$ and define
\[
\p_N := \frac{1}{\sqrt{2}\sigma_N}1\!\!1_{[-\frac{1}{N},\frac{1}{N}]\cap \textrm{sprt}\rho },
\]
where, as in the introduction, $\sigma^2_N=\rho([0,1/N])$.
Then, $\p_N\in \mathcal{L}^2_\rho$ and $\|\p_N\|_{\mathcal{L}^2_\rho}=1$.
Write $\psi_N(t) =\int_{T^*} e^{-i\lm t}\p(\lm) d\rho(\lm)$.
Claim \ref{clm: one} implies that
\[
f(t)\overset d=\zeta \psi_N (t)\oplus R(t)
\]
where $\zeta\sim N(0,1)$ and $R$ is a centered Gaussian (not necessarily stationary) process.
Thus, for any $\ell>0$ we have:
\begin{equation}\label{eq: plug low}
P_f(N)\geq \Pro\Big(\zeta \psi_N(t)\geq \ell, \ \forall t\in [0,N]\Big)\cdot \mathbb{P}\Big(|R(t)|\leq \ell,  \ \forall t\in [0,N]\Big).
\end{equation}

To estimate the first term, we note that for $t\in[0,N]$ the function $\psi_N$ satisfies
\[
\psi_N(t)=\frac{1}{\sqrt 2 \sigma_N}\int_{0}^{\frac{1}{N}}\cos (t\lm) d\rho(\lm)
\geq \frac{1}{\sqrt{2} \sigma_N}\cos\left(\frac{t}{N}\right) \sigma^2_N\geq \frac{1}{2\sqrt{2}}\sigma_N.
\]
Therefore,
\[
\mathbb{P}(\zeta \psi_N(t)\geq \ell, \forall t\in [0,N])\geq \mathbb{P}\Big(\sigma_N\zeta \geq 2\sqrt{2} \ell\Big).
\]

For the second term, we use Lemma \ref{lem: anderson} to get:
\[
\mathbb{P}(|R(t)|\leq \ell, \ t\in [0,N])\geq \mathbb{P}(|f(t)|\leq \ell, \ t\in [0,N]).
\]
We now apply ball estimates -- Lemma \ref{lem: KS} over $\Z$, and Lemma \ref{lem: prod} over $\R$ --  to obtain:
\[
\mathbb{P}\Big(|f(t)|\leq \ell, \ t\in [0,N]\Big)\geq \mathbb{P}\Big(\beta|f(0)|\leq \ell\Big)^N,
\]
where over $\Z$ the above is valid for all $\ell>0$ and with $\beta=1$, while over $\R$, as $\var f(0)=1$, it is valid with $\ell>\ell_0(m_\delta, \delta)$ and a certain $\beta=\beta(\delta,m_\delta)$.
By plugging our estimates back into \eqref{eq: plug low} and substituting $\tilde{\ell}=2\sqrt{2}\ell$, the result follows.
\end{proof}

\subsection{Explicit lower bounds: proof of Proposition \ref{cor: low exp} }\label{sec: explicit low1}

Let $\ell_0=0$ and $\beta=1$ over $\Z$, while $\ell_0=\ell_0(m_\delta, \delta)$ and $\beta=\beta(\delta,m_\delta)$ over $\R$ as are given in Theorem~\ref{thm: balance low}. Applying Lemma \ref{lem: tail}(a) to Theorem \ref{thm: balance low} gives the following estimate for all $\ell>\max\{\ell_0, 2\si_N\}$ and $N>0$:
\begin{equation}\label{eq: low balance}
 P_f(N)\geq \exp\left(- {\ell^2}/{ \si_N^2}\right)\cdot \mathbb{P}(\beta|Z|\leq \ell)^N.
\end{equation}
By our premise, there are some $b,\g>0$ such that $\si_N^2 \ge b N^{-\g}$ along a subsequence of $N$. Let $N$ be a member of that subsequence.
We will choose the level $\ell=\ell(N)$ and estimate the terms in \eqref{eq: low balance} in each of three cases.

\medskip
{\bf Case 1: spectrum exploding at $0$ ($\g<1$).} Put $\ell=\ell(N)=\beta\sqrt{2\log N}$, then there exists $N_0$ such that for $N>N_0$ we have $\ell(N)>\ell_0$. This yields for the first term
\[
\exp\left(- { \ell^2}/{ \si_N^2}\right)\geq e^{-CN^{\gamma}\log N},
\]
while by Lemma \ref{lem: tail}(a) we have for the second term
\[
\Pro\Big(\beta|Z|\leq \ell\Big)^N\geq \Big(1-2e^{-\log N}\Big)^N= \left(1-\tfrac{2}{ N}\right)^N\ge e^{-2}.
\]

\medskip
{\bf Case 2: spectrum bounded near $0$ ($\g=1$).}
Fix an arbitrary $\ell>\max\{\ell_0,\beta\}$. Then
\[
\exp\left(- { \ell^2}/{ \si_N^2}\right)\geq e^{- \ell^2 N/b},
\]
while
\[
\mathbb{P}(\beta|Z|\leq \ell)^N\geq \mathbb{P}(|Z|\leq 1)^N=e^{-cN}.
\]

\medskip
{\bf Case 3: spectrum vanishing at $0$ ($\gamma>1$) over $\Z$.} Put $\ell^2=8 N\si^2_{N} \ge 8bN^{1-\g}$. The first term is
\[
\exp\left(- {\ell^2}/{ \si_N^2}\right)= e^{- 8 N}.
\]
Using Lemma \ref{lem: tail}(b) we bound the second term:
\[
\Pro\Big(2\sqrt{2}|Z|\leq \ell\Big)^N\geq\Pro\Big(|Z|\leq \sqrt{b} N^{(1-\g)/2} \Big)^N \ge
\Big(\tfrac{ \sqrt b }{ 4 } N^{(1-\g)/2} \Big)^N
\geq e^{-C N\log N}.
\]
In all cases, the estimate stated in Proposition \ref{cor: low exp} follows.

\subsection{Vanishing spectrum over \texorpdfstring{$\Z$}{Z}: proof of Proposition \ref{cor: low} }\label{sec: explicit low2}

Over $T=\Z$, assume that $N\si_N^2 <1$.
We imitate the proof of the case $\g>1$ of Proposition \ref{cor: low exp},
and set $\ell^2= 8 N \si_N^2$. We estimate both parts of \eqref{eq: low balance}.
As above, the first term is estimated by
$ e^{- 8 N}. $
By Lemma~\ref{lem: tail}(b)
the second term satisfies
 \[
\Pro\Big(2\sqrt{2}|Z|\leq \ell\Big)^N =\Pro\Big(|Z|\leq \si_N\sqrt{N} \Big)^N \ge
e^{ C N \left( \log(N\si_N^2)-1\right) }.
\]
The estimate follows.

\section{Upper bounds }\label{sec: up}

The main result of this section is a general upper bound on the persistence probability.

Assume that $f$ is a GSP whose spectral measure $\rho$ satisfies $\rho_{ac}\ne 0$.
This implies that there exist $\nu>0$ and $E\subseteq \R$ s.t. $E$ is a bounded set of positive measure on which $d\rho_{ac} \ge \nu\, dx$. We let $q$ be such that $\frac{1}{q}E\subseteq [-\pi,\pi]$.
Further, assume that $\g\ge 0$ is such that
$m_{-\g}(\rho)<\infty$ and let $k\in\N\cup \{0\}$ and $0\leq s<2$ be such that $\g=2k+s$. Put $r=\max\{k, s/2\}$.
We have the following.

\begin{thm}[general upper bound]\label{thm: balance up}
There exist universal positive constants $c_0$, $c_1$, and a constant $c(s)$ depending only on $s$, such that the following holds; let $f$ be a GSP whose spectral measure $\rho$ obeys \eqref{eq: basic cond} and has a nontrivial absolutely continuous component. Let $E,\nu,q,\g, k, s, r $ be as described above. Set
\begin{equation*}
\alpha =c_0 |E|, \quad
\beta =
\begin{cases}
(c_1 k)^{-k}\sqrt{ \frac{\nu|E|}{m_{-2k}} }, & k> 0 \\
c(s)\sqrt{ \frac{\nu|E|}{m_{-s} }}, & k=0
\end{cases}
\quad \text{and}  \quad
\theta = \begin{cases}
 \sqrt{\frac{m_{-2k+2}{(\rho)}}{4 m_{-2k}{(\rho)}}}, & k>0 \\
 1, & k=0.
 \end{cases}
\end{equation*}
Then there exists $N_0(E)>0$ such that for every $N>\max\{N_0,2k\}$ and
\[\ell_0(N)=
2N^{-r} \max\{\sqrt{\log (\theta N)},1\}
\]
the following holds for all $\ell > \ell_0(N)$:
\[
P_f(N) \le 2\, \Pro\big( N^{-r} Z >  \ell\big) + {2qN}\, \Pro\big( \beta |Z| <  \ell\big)^{\al N}.
\]

\end{thm}

We make two additional remarks which we establish at the end of the proof of Theorem~\ref{thm: balance up}.
\begin{rmk}\label{rmk: N0}
Over $\R$, if $E$ contains an interval $J$, then one can take $N_0 = \frac{2\pi }{|J|}$.
\end{rmk}

\begin{rmk}\label{rmk: thm relax}
In case $\g<2$, one may replace the condition $m_{-\g}<\infty$ with $\si_N^2 \le bN^{-\g}$ for all $N>0$ and some $b>0$.
In this case $\beta = c(s) \sqrt{\frac{\nu|E|}{2b}}$, while other constants remain unchanged.
\end{rmk}

Theorem \ref{thm: balance up} gives a recipe for estimating $P_f(N)$: given $N$ and $\rho$, one should choose a level $\ell=\ell(N,\rho)$ so that
the factors in the theorem's estimate are of the same order.
This recipe is used to derive Propositions \ref{cor: up}, \ref{cor: up inf} and \ref{cor: small}.

We will use the following reduction.

\begin{clm}\label{clm: E in pi}
It is enough to prove Theorem \ref{thm: balance up} assuming that condition \eqref{eq: AC cond} holds for $E \subseteq [-\pi,\pi]$.
\end{clm}

\begin{proof}
Over $\Z$ the claim is trivial. Let $T=\R$.
Suppose that Theorem \ref{thm: balance up} holds when $E\subseteq[-\pi,\pi]$. In the general case, let $\rho= \nu\ind_E + \mu$ for some ${\nu}>0$ and measurable $E\subset \R$. Set the constants $\alpha$ and $\beta$, as well as the function $\ell_0(N)$, as defined in Theorem \ref{thm: balance up}.
Let $q>1$ be such that $\tfrac 1 q E\subseteq [-\pi,\pi]$ {and let} $\tilde{f}$ be the GSP defined by $\tilde{f}(x)=f(\tfrac{x}{q})$. Then $\tilde{f}$ has spectral measure $\tilde\rho = q \nu \ind_{\tfrac 1 q E} + \tilde \mu$
whose moments satisfy $m_{\g}(\tilde \rho)={q^{-\g}} m_{\g}(\rho)$ for any {$\g\in\R$}.
Thus the corresponding {values of}
$\tilde{\al}$, $\tilde{\beta}$ and
 $\tilde{\ell_0}(N)$ satisfy: $\tilde{\alpha}=q^{-1}\al$,
$\tilde{\beta}=q^{-r}\beta$ and
$\tilde{\ell_0}(qN) = q^{-r}\ell_0(N)$, where as above, $r=\max\{k,s/2\}$.  Set $N_0(\tfrac 1 q E)$ as defined in Theorem \ref{thm: balance up} and let $N_0(E)=N_0(\tfrac 1 q E)$. Note that since $q>1$, if $N>\max\{N_0(\tfrac 1 q E), k\}$ then $qN$ satisfies the same condition. So, given $N>\max\{N_0(\tfrac 1 q E), k\}$ and
$\ell>\ell_0(N)$
we may apply Theorem \ref{thm: balance up} for $\tilde{f}$, with $\tilde{N}=qN$ and $\tilde{\ell}=q^{-r}\ell > \tilde{\ell_0}(\tilde{N})$, to get:
\begin{align*}
P_f(N) &= \Pro\Big(f(t)>0 \ \forall t\in [0,N] \Big)
= \Pro\Big(\tilde{f}(t)>0 \ \forall t\in [0,qN]\Big)
\\ & \le \Pro \left( (qN)^{-r} Z > q^{-r}\ell \right) +{2qN}\, \Pro\left(q^{-r}\beta |Z| \le q^{-r}\ell \right)^{\tfrac \alpha  q\cdot qN}
\\ & = \Pro \left( N^{-r} Z > \ell
 \right) + {2qN}\,\Pro\left(\beta|Z| \le \ell \right)^{\alpha N}.
\end{align*}
\end{proof}

\subsection{General upper bound: proof of Theorem \ref{thm: balance up} }\label{sec: vanish up}

We turn to prove Theorem \ref{thm: balance up}.
We will give full details for the case $T=\R$, as the proof for $T=\Z$ is almost identical. Throughout the proof we denote by $C$ universal constants which may change from line to line, and by $C(s)$ constants depending only on $s$, which again, may change from line to line. We divide the proof into several steps.

\medskip
\underline{\textbf{Step I:}} Let $\g>0$ be such that $m_{-\g}<\infty$ and write $\g = 2k+s$, where $k\in \N\cup \{0\}$ and $0\le s<2$.
If $k\neq 0$ then, by applying Observation~\ref{obs: integ} $k$ times, we find that there exists a GSP $F_{k}$ which satisfies $F_{k}^{(k)}\overset{d}{=}f$. We denote the spectral measure of $F_{k}$ by $\mu_{k}$ and note that, by Observation~\ref{obs: rate},
$d\mu_k= \frac{d\rho(\lm)}{\lm^{2k}}$. We therefore have, due to stationarity,
\[
\var(F_k(t)) =m_{-2k}(\rho) \qquad \forall{t\in\R},
\]
and, due to Lemma \ref{lem: sup stat},
\[
\quad \E\Big(\sup_{[0,N]} F_k \Big) \le  C\sqrt{m_{-2k}{(\rho)}\max\{\log \left(\theta_k N\right),1\}},
\]
where $\theta_k :=\sqrt{\frac{m_{-2k+2}{(\rho)}}{4 m_{-2k}{(\rho)}}}$.

If $k=0$, we integrate one time to get the process $F_{\frac{s}{2}}(t) :=
 \int_0^t f(\tau) d\tau$, so that $F_{\frac{s}{2}}'\overset{d}{=}f$. Notice that in this case, $F_{\frac{s}{2}}$ is a Gaussian process, but not necessarily stationary.
From Observation \ref{obs: IBP} we deduce that $\rho([0,\lm]) \le \frac{1}{2}m_{-s} {(\rho)}\lm^{s}$ for all $\lm>0$.
 Since we assume that $m_0=1$, Lemma \ref{lem: average} implies that, for every $N>0$,
\[
\sup_{t\in [0,N]}\var (F_{\frac{s}{2}}(t)) \le C(s) m_{-s}{(\rho)}N^{2-s}.
\]
while Lemma \ref{lem: sup H} implies that, for every $N>1$, we have
\[
\E \Big(\sup_{[0,N]} F_{\frac{s}{2}}(t) \Big)\le C(s) \sqrt{m_{-s}{(\rho)}} N^{1-\tfrac s 2}.
\]
Write $\theta_0=1$ and recall that $r=\max\{k,s/2\}$. For $\ell >2N^{-r} \max\{\sqrt{\log (\theta_k N)},1\}$ set
\[
M(N)= \E \Big(\sup_{[0,N]} F_r\Big)+\sqrt{2\sup_{t\in[0,N]}\var (F_r(t))}N^r \ell.
\]
Our estimates yield
\[
M(N)\le \begin{cases}
 C\sqrt{m_{-2k}}(\max\{\sqrt{\log (\theta_k N)},1\} + N^k \ell) , & k>0 \\
C(s) \sqrt{m_{-s}} N^{1-\tfrac s 2}(1+N^{s/2} \ell)  , &k=0
\end{cases},
\]
or rather, due to the condition on $\ell$,
\begin{equation}\label{eq: bounds on M}
M(N)\le \begin{cases}
  C\sqrt{m_{-2k}} N^k \ell , & k>0 \\
C(s) \sqrt{m_{-s}} N \ell, &k=0.
\end{cases}
\end{equation}

\medskip
\underline{\textbf{Step II:}} Consider the event
\begin{equation*}
G =\Big\{ \sup_{[0,N]} \left|F_r\right| \le M(N) \Big\}. 
\end{equation*}
We will estimate the persistence probability through:
\begin{equation}\label{eq: p12}
P_f(N) \le \Pro(G^c) + \Pro( \{f> 0 \text{ on } [0,N]\} \cap G )=:P_1 + P_2.
\end{equation}

To estimate $P_1$, we apply the Borell-TIS inequality (Lemma \ref{lem: BTIS}). Using also the fact that $F_r$ and $-F_r$ are identically distributed, we get 
\begin{align}\label{eq: Gc}
P_1&= {2} \,\Pro\left( \sup_{[0,N]} F_{r} >
 \E \big(\sup_{[0,N]} F_r\big)+\sqrt{2\sup_{[0,N]}\var (F_r)}N^r \ell
 \right)  \notag
 \\ &\le {2}\, e^{-N^{2r}\ell^2} \le 2\, \Pro\left(N^{-r} |Z| > \ell\right),
\end{align}
where in the last step we use Lemma \ref{lem: tail}(a) and the fact that $\ell N^r>2$.

\medskip

\underline{\textbf{Step III:}} We turn to the estimate of $P_2$. A simple translation yields
\[
P_2 =  \Pro\left( \Big\{\widetilde{f}> 0 \text{ on } [-\tfrac N 2,\tfrac N 2 ] \Big\} \cap
\Big\{\sup_{\left[- N /2, N/ 2 \right]} \widetilde{F}_r \le M(N) \Big\}  \right),
\]
 where $\widetilde{f}(x) = f(x+\tfrac N 2)$ has the same distribution as $f$, and $\widetilde{F}_r := F_r(x+\tfrac N 2)$ obeys the same differential relation as $F_r$ (namely, $F_r^{(r)} \overset{d}{=} f$ if $r=k\ge 1$ and $F_{r}' \overset{d}{=} f$ if $r<1$).
Since $k\le \frac{N}{2}$ we may apply Theorem~\ref{thm: anal} to get
\begin{align} \label{eq: sum on R}
P_2\le \Pro \left( \Big\{f> 0 \text{ on }  [-\tfrac N 2,\tfrac N 2 ]\Big\} \cap
 \Big\{\frac 1 N \int_{-\tfrac 9{40} N}^{\tfrac 9{40} N } f \le  L\Big\}  \right) ,
\end{align}
where, by \eqref{eq: bounds on M},
\begin{equation} \label{eq: L}
L =
\begin{cases}
(Ck)^k  \sqrt{ m_{-2k}}  \, \ell & k> 0 \\
C(s)\sqrt{m_{-s}}\, \ell, & k=0.
\end{cases}
\end{equation}
Set $I_N=[-\tfrac 9{40} N , \tfrac 9{40} N-1 ]\cap \Z$. Observe that if $f>0$ and $\frac 1 N \int_{-9N/40}^{9N/40}f \le  L$, then
\[
\Big|\Big\{ v\in [0,1]:  \tfrac{1}{N} \sum_{n\in I_N} f(n+v) <2L\Big\}\Big|\ge \frac 1 2.
\]
Let $v\sim \text{Unif}([0,1])$ be a uniform random variable which is independent of $f$. On the product of the probability spaces of $f$ and $v$ define the events
\[
V = \Big\{f(n+v)> 0, n\in I_N\Big\}\cap \Big\{\frac{1}{N} \sum_{n\in I_N} f(n+v) <2L\Big\},
\]
and
\[
U = \Big\{f(t)> 0, t\in [0,N] \Big\}\cap \Big\{ \frac 1 N\int_{-\tfrac 9{40} N}^{\tfrac 9{40} N } f \le  L\Big\}.
\]
Then
\[
\Pro_{f,v}\big(V \big) \ge \Pro_{f,v}\big(V \mid U\big) \Pro_{f}\big(U\big) \ge \frac 1 2 \Pro_{f}\big(U\big).
\]
The estimate in \eqref{eq: sum on R} and the stationarity of $f$ now imply that
\begin{equation}\label{eq: split}
P_2 \le 2\Pro_{v,f} \big(V \big) \le
2\Pro_{f}\Big(\Big\{f(n)> 0, n\in I_N\Big\}\cap \Big\{\tfrac{1}{|I_N|} \sum_{n\in I_N} f(n) <6L\Big\}\Big).
\end{equation}

\medskip

\underline{\textbf{Step IV:}} By Claim \ref{clm: E in pi} we assume, without loss of generality, that $d\rho(\lm)\ge \nu\ind_E(\lm)d\lm$ for $\nu>0$ and $E\subseteq [-\pi,\pi]$.
By Claim \ref{clm: Riesz} used with $\ep=\tfrac 1 2$, there exist $\Lambda=\{\lm_n\}\subseteq\Z$  of density $A=\tfrac 1 {4\pi} |E|$ and a number $B= \sqrt{C \nu |E|}$ such that
\begin{equation}\label{decompazition}
f(\lm_n)\overset d = B Z_n \oplus g_n, \quad \text{ where } \{Z_n\} \text{ are i.i.d. }\calN (0,1).
\end{equation}

 By the definition of the density $D^{-}(\Lambda)$, there exists $N_0$, depending only on $E$, such that for all $N>N_0$ we have $|\Lambda\cap I_N|>A|I_N|/2$. From this point we assume $N$ to satisfy this condition. Denote
 \begin{equation}\label{eq: d}
 d=\Big\lfloor  \frac{A|I_N|}{4}\Big\rfloor,
 \end{equation}
 where $\lfloor a\rfloor$ is the integer value of $a$. Let $\Lambda_N$ be the set containing the smallest $2d-1$ elements of $\Lambda\cap I_N$.

For $\tau\in [0,|I_N|-1]\cap\Z$  consider the two disjoint sets $(\Lambda_N+\tau)\cap I_N$ and $(\Lambda_N+\tau-|I_N|)\cap I_N$.  One of these sets has more elements then the other and, in particular, at least $d$ elements. Let $\widetilde{S}(\tau)$ be the first $d$ elements of that set. Further, let $S(\tau)$ be the corresponding translate of $\widetilde{S}(\tau)$ by either $\tau$ or $\tau-|I_N|$ so that $S(\tau)\subseteq \Lambda_N$.

If $\frac{1}{|I_N|} \sum_{j\in I_N} f(j) <6L$ then, by Claim \ref{average}, there exists $\tau\in [0,|I_N|-1]$ so that,
\[
\frac{1}{|\Lambda_N|} \sum_{j\in (\Lambda_N+\tau)\cap I_N} f(j)+\frac{1}{|\Lambda_N|} \sum_{j\in (\Lambda_N+\tau-|I_N|)\cap I_N} f(j) <6L.
\]
Recalling that $|\Lambda_N|= 2d-1<2d$, this implies, in particular, that
\[
\frac{1}{d} \sum_{j\in \widetilde{S}(\tau)} f(j)<12 L.
\]
This allows us to apply a simple union bound to the expression in \eqref{eq: split} and find that
\[
\begin{aligned}
P_2  &\le 2\sum_{\tau= 0}^{|I_N|-1}\Pro \Big(\Big\{f(j)> 0, j\in \widetilde{S}(\tau)\Big\}\cap \Big\{\frac{1}{d} \sum_{j\in \widetilde{S}(\tau)} f(j) < 12 L\Big\}\Big)\\
&= 2\sum_{\tau= 0}^{|I_N|-1} \Pro \Big(\Big\{f(j)> 0, j\in S(\tau)\Big\}\cap \Big\{\frac{1}{d} \sum_{j\in S(\tau)} f(j) <12 L\Big\}\Big),
\end{aligned}
\]
where the last step is due to the stationarity of the process $f$.

Let
\[
 \Sigma := \Big\{\underline{x}:=(x_1,\dots,x_d): \:  x_1,\dots,x_d \geq 0, \ \frac 1 d \sum_{n=1}^d x_n \leq
12 L\Big\}\subset \R^d,
\]
and denote $\underline{Z}:=(Z_1,...,Z_d)$  where the $Z_i$ are i.i.d. random variables with distribution $\calN (0,1)$. Since {for every $\tau$}, $S{(\tau)}\subset \Lambda$ and $|S(\tau)|=d$, the decomposition in \eqref{decompazition}, combined with the estimate $|I_N|<N$, implies that
\begin{equation}\label{sigma-ineq-1}
P_2 \leq 2N \sup_{\underline{g}\in\R^d}\Pro_{Z} \Big(B\underline{Z}\in  \Sigma-\underline{g}\Big),
\end{equation}

\medskip

\underline{\textbf{Step V:}}
We claim that,
\begin{equation}\label{sigma-ineq-2}
\sup_{\underline{g}\in\R^d}\Pro_{Z} \Big(B\underline{Z}\in  \Sigma-\underline{g}\Big) \le   \Pro\left(B |Z_1| \le C L\right)^{d}
 \end{equation}
 where $C>0$ is a universal constant. To this end we consider two cases.

\smallskip
{\bf Case 1: ${144}L B^{-1}< 1$}. For a fixed $\underline{g}\in\R^d$ we have
\begin{align*}
\Pro_{Z} \Big(B\underline{Z}\in  \Sigma-\underline{g}\Big)  &\le B^{-d} \left|\Sigma-\underline{g}\right| = B^{-d}|\Sigma| =  \frac {(12 \,LdB^{-1})^d }{d!}
\\&\le (36 \,L B^{-1})^d
 \le \Pro\left(B |Z_1| \le 144 \,L \right)^{d},
\end{align*}
where the last step holds by Lemma \ref{lem: tail}(b).

\smallskip
{\bf Case 2: $144 \, LB^{-1}\geq 1$}.
 For a fixed $\underline{g}\in\R^d$ we first note that if $ \frac 1 d\sum_{n=1}^d g_n >12 \, L$, then the shifted simplex $\Sigma-\underline{g}$ does not contain $\underline{0}$. Therefore, in that case, there exists another shift $\underline{h}$, with $ \frac 1 d\sum_{n=1}^d h_n\leq 12 \, L$, such that $ \Pro_Z( B \underline{Z} \in \Sigma-\underline{g})\leq \Pro_Z( B \underline{Z} \in \Sigma-\underline{h}) $ (one may take $h$ to be the point where $\norm{x}$ attains its minimum on $\Sigma-g$, noticing that the density of $Z$ is monotone decreasing in $\norm x$).
We can therefore assume, without loss of generality, that $ \frac{1}{d}\sum_{n=1}^d g_n \leq 12 \, L$. By Claim~\ref{clm: iid} we have
\[
\Pro\left( B \underline{Z}\in \Sigma -\underline{g} \right)  \le  \Pro\left(B Z_n + g_n {\geq}0, n=1,...,d \right) \le \Pro\left( B Z_1  {\leq} 12 \,L \right)^d.
\]
Now, Applying Claim~\ref{clm: tails comp} with $ \delta=1/ 12 $ we get, for some (universal) $\eta>0$, that
$\Pro\left( B Z_1  {\leq 12 \,L} \right) < \Pro\left (B |Z_1|\le 12 \,\eta L\right)$. This establishes \eqref{sigma-ineq-2} with $C=\max(144, 12\, \eta)$.

Inserting \eqref{sigma-ineq-2} into \eqref{sigma-ineq-1} we find that
\begin{equation}\label{sigma-ineq-3}
P_2\leq 2N\Pro\left (B |Z_1|\le CL\right)^d.
\end{equation}

\medskip
\underline{\textbf{Step VI:}} We insert the estimates \eqref{eq: Gc} and \eqref{sigma-ineq-3} into \eqref{eq: p12}, to find that under the conditions of the theorem,
\begin{equation}
\begin{aligned}
P_f(N)\leq  &2\, \Pro\left(N^{-r}Z>\ell \right)+ 2N\,\Pro\left (B |Z_1|\le CL\right)^d\\
=&{2}\, \Pro\left(N^{-r}Z>\ell \right)+ 2N\, \Pro\left(\beta |Z_1|\le \ell\right)^{\alpha N},
\end{aligned}
\end{equation}
where due to \eqref{eq: L} and the fact that $B=\sqrt{C\nu|E|}$,
\[
\beta=
\begin{cases}
(Ck)^{-k}\sqrt{\tfrac{\nu|E|}{m_{-2k}}}, & k> 0 \\
C(s)\sqrt{\tfrac{\nu|E|}{m_{-s}}}, & k=0,
\end{cases}
\]
while due to \eqref{eq: d},
$\alpha=C|E|$, for some universal constant $C$. This completes the proof of the theorem.

\medskip
We end this section by proving Remarks \ref{rmk: N0} and \ref{rmk: thm relax}.

\smallskip
\emph{Proof of remark \ref{rmk: N0}:}
Recall that $N_0$ was such that for $N>N_0$ and any interval {$I=[-N/2,N/2]$} we have $|\Lambda\cap I|>D^{-}(\Lambda) \, |I|/2$, where $\Lambda$ is the set from Theorem \ref{thm: ver}.
Now, if $E$ is {contains} an interval $J$, one can directly check that the lattice $\Lambda_E = \tfrac{2\pi}{|J|}\Z$ obeys the properties in Theorem \ref{thm: ver}. Therefore, indeed, we may take $N_0 = \frac{2\pi}{|J|}$.

\smallskip
\emph{Proof of remark \ref{rmk: thm relax}:
We note that the condition $m_{-\g}<\infty$ was used only twice: first, for $k\ge 1$, in order to apply Observation \ref{obs: integ} and obtain stationarity of $F_k$, and second, for all $k$,
in order to apply Observation \ref{obs: IBP} and obtain $\rho([0,\lm]) \le \frac{1}{2}m_{-s} {(\rho)}\lm^{s}$ for all $\lm>0$.
Therefore, in case $k=0$ (that is, $\g<2$), we may directly assume that $\rho([0,\lm]) \le b \lm^{s}$ for all $\lm>0$ and some $b>0$.
The proof goes through, with $m_{-s}$ replaced by $2b$ everywhere.
}

\subsection{Explicit upper bounds: proof of Proposition \ref{cor: up} }
In this section we prove Proposition~\ref{cor: up}
by applying Theorem \ref{thm: balance up}. 
Firstly, we note that we may apply Theorem \ref{thm: balance up}, even though our assumption 
$\si_N^2 < b N^{-\g}$ 
is formally slightly weaker than the assumption $m_{-\g}<\infty$ of the theorem (see \eqref{eq: IBP items}). 
Indeed, for $\g\le 1$, this is possible, by Remark~\ref{rmk: thm relax};
while for $\g>1$, 
we use \eqref{eq: IBP items} to deduce that
$m_{-\g'}<\infty$ for all $\g' \in (1,\g)$, and then apply Theorem~\ref{thm: balance up} with $\g$ replaced by some $\g'\in(1,\g)$  (the form of the upper bound will not be effected by this change). For simplicity, we keep using the letter $\g$ (and not $\g'$) also in the latter case.

In each case, we must choose an appropriate $\ell=\ell(N,\beta,\g)$ in Theorem \ref{thm: balance up}. We leave it to the reader to verify that in all cases our choice satisfies
\begin{equation}\label{eq:cond verify} \ell \geq 2 N^{-r}\max\left(\sqrt{\log (\sqrt{\tfrac{m_{-2k+2}}{4 m_{-2k}}} N)},1\right)
\end{equation}
 for large enough $N$ (depending on $\beta,\g$), thus fulfilling the requirements of the theorem.
 
\smallskip
{\bf Case 1: spectrum exploding at $0$ ($0<\g <1$)}.
Set $\ell>0$ so that
$e^{-\ell^2/\beta^2} = \frac{\log N}{N^{1-\g}} $.

{We note that, as \eqref{eq:cond verify} is satisfied, we may apply Lemma \ref{lem: tail}(a).} 
Using this and the inequality $\log(1-x) \le -x$
we obtain {for some positive constant $c_1$},
\begin{align*}
P_f(N) &\le
2\, \Pro(N^{-\g/2 } Z > \ell) + 2{q}N\, \Pro\left( \beta |Z|\le  \ell \right)^{\al N}
\\ & \le2  e^{- N^\g \ell^2 /2 } + 2{q}N\left( 1-2e^{- \ell^2/\beta^2  } \right)^{\al N}
\\ & \le 2 e^{- N^\g \ell^2 /2 }+ 2{q}N e^{-\al N \cdot 2 e^{-\ell^2/\beta^2}}
\\ & \le 2 e^{-c_1  N^\g \log N}  + 2{q}N e^{-2\al N^\g \log N} \le  \exp\left( - C N^\g \log N\right).
\end{align*}

\smallskip
{\bf Case 2: spectrum bounded near $0$ ($\g=1$).}
In this case we choose $\ell=1$ (or any other constant) to obtain the exponential bound:
\[
P_f(N) \le 2\, \Pro(Z>\sqrt N) + 2 qN \, \Pro \Big(\beta|Z|\le 1\Big)^{\al N} \le e^{-C N}.
\]

\smallskip
{\bf Case 3: spectrum vanishing at $0$ ($\g>1$).}
Again we note that, as \eqref{eq:cond verify} is satisfied, we may apply Lemma~\ref{lem: tail}.
Using both parts of this lemma together with Theorem \ref{thm: balance up} we get that
\begin{equation}\label{eq: up to open}
P_f(N)  \le
2\, \Pro( N^{-r } Z > \ell) + 2qN\,\Pro\left( \beta |Z|\le  \ell \right)^{\al N} \le 2e^{-N^{2r} \ell^2/2 }+ 2q N \, \left( \frac{\ell^2}{\beta^2} \right)^{\al N/2}.
\end{equation}

Setting
$\ell^2= \beta^2 (2r-1) N^{1-2r} \log N$ we get {for some constant $c_1$,}

\begin{align*}
P_f(N) &
\le 2e^{-\tfrac 1 2 \beta^2 (2r-1) N \log N } + 2qN \, e^{-{c_1} (2r-1) N \log N}
\\ &  \le \exp(-C N\log N).
\end{align*}

In all cases, the estimate holds for $N>N_0(E, \al,\beta,\g)$ and $C(\al,\beta,\g)$.
Moreover, $C$ is proportional either to $\al$ or $\beta^2$, thus is linear in $|E|$.

\subsection{Vanishing of infinite order: proof of Proposition \ref{cor: up inf}}\label{sec: inf order}

For a given $N>0$ we may apply Theorem \ref{thm: balance up} with $\g =2k$, so long as $m_{-2k}<\infty$, $k \le  \frac{N}{2}$ and our chosen level $\ell(N)>\ell_0(N)$. We will choose $k=k(N)$ later.
Denote $\tau :=\sup \{ \sqrt{ m_{-2k+2} / m_{-2k}} : \: k \in \N_0, m_{-2k}<\infty\}$,
and note that $\tau<\infty$ by Claim \ref{clm: finite tau}. Let
\begin{equation}\label{eq: ell inf}
\ell(N)=  N^{-k}\cdot \sqrt{c_0 |E| k N \log \frac{\tau N}{c_1 k}},
\end{equation}
where all the constants ($c_0, c_1, |E|$) are as in Theorem \ref{thm: balance up}.
Notice that indeed, this choice obeys $\ell(N)>\ell_0(N)$, for every large enough $N$ and $k \le \frac{N}{2}$.
By \eqref{eq: up to open} and the explicit forms of $\al, \beta$ and $\ell$ we get:
\begin{align}\label{eq: explicit up}
P_f(N) & \notag\le 2\exp(-N^{2k} \ell^2/2) + 2q N \exp\big(\al N (\log \ell-\log \beta) \big)
\\ & \le 2\exp\left(- \frac{c_0}{2}|E| k N \log \frac{\tau N}{c_1 k} \right)
\\ & \notag\: + 2q N \exp\left\{ c_0|E| N   \left( -k\log\frac{N }{c_1 k} +\frac 1 2 \log \frac{c_0 m_{-2k}}{\nu}  + \frac 1 2\log \left( k N \log \frac {\tau N}{c_1 k}\right) \right) \right\}
\end{align}
Now, we use $k=k(N)\in \N_0$ which satisfies
\begin{equation}\label{eq: k}
 c_1 k\left( \frac{c_0  m_{-2k} }{\nu  }   \right)^{1/k}\le  N , \quad \text{and}\quad  k\le \frac{N}{2}.
\end{equation}
For this $k$ we have $\displaystyle \frac 1 2 \log \frac{c_0 m_{-2k}}{\nu  } \le \frac k 2 \log \frac N {c_1 k}$, so \eqref{eq: explicit up} becomes

\[
\begin{aligned}\label{eq: inf end}
P_f(N)&\le 2\exp\left(- {\frac{c_0}{2}}|E| k N \log \frac{\tau N}{c_1 k} \right)\\&\hspace{10pt}
+ 2q N \exp\left\{ c_0|E| N   \left( -\frac{k}{2}\log\frac{N }{c_1 k}  + \frac 1 2\log \left( k N \log \frac {\tau N}{c_1 k}\right) \right) \right\}\\
&\le \exp\left(- C N k  \log \frac{cN}{ k} \right),
\end{aligned}
\]
for some constants $C$ and $c$ depending on $\rho$. 
Finally, we note that, if $\nu$ and $E$ do not vary with $N$, one may replace \eqref{eq: k} with the choice of any integer $k=k(N)$ satisfying
\[
1\le k\le \min\left(\tfrac 1 2,c\right) N, \quad k \,m_{-2k}^{1/k}\le  N.
\]
This will effect our bound only by a multiplicative constant in the exponent.
The specific examples follow easily.

\subsection{Tiny persistence: proof of Proposition \ref{cor: small}}\label{sec: small}

Here we prove Proposition \ref{cor: small} from Theorem \ref{thm: balance up}. Let $w(\lm)$ be the density of the
absolutely continuous component of the spectral measure. Condition \eqref{eq: AC cond} becomes $w \ge \nu \ind_E$ for some $E$ of positive measure and $\nu>0$. The proof will follow that of Proposition \ref{cor: up inf}, however this time choose $E=[1,x_N]$ and a suitable ${\nu_N}$, both of which depend on $N$.
We note that, by Remark \ref{rmk: N0}, {our} estimates are valid for {all} $N>1$ (provided that $x_N$ is big enough).
We shall choose $k=k(N)$ that satisfies \eqref{eq: k} and $\ell=\ell(N)$ as in \eqref{eq: ell inf}.
It remains to optimize the choice of $x_N$.

\smallskip
{\bf Part 1.}
Since $w(\lm)=0$ for $|\lm|\le 1$, we have $m_{-k} \le 1$ for all $k>0$.
Also, $w(\lm)\ge x_N^{-\al}$ for $\lm\in [1,x_N]$. Putting $\nu_N=x_N^{-\al}$, \eqref{eq: k} becomes
\[
c_1 c_0^{1/k}k x_N^{{\al}/k} \le N  \quad \text{and}\quad  k\le \tfrac{N}{2}.
\]
By choosing $x_N = e^{\frac{cN}{\al}} $ and $k= cN$,  for an appropriate universal constant $c>0$, we satisfy these inequalities. We may therefore conclude by \eqref{eq: explicit up}
that for large enough $N$,
\[
P_f(N) \le {3Nx_N}\exp(-C x_N \cdot N^2 ) \leq \exp(- e^{{C_1} N}).
\]

\smallskip
{\bf Part 2.}
In this case we have $k=2$ and $\nu=x_N^{-\al}$. Thus
\eqref{eq: k} becomes $c x_N^{\al/2} \le N$, for some constant $c>0$.
 This is satisfied by the choice $x_N = \left(\frac{N}{c}\right)^{1/\al}$, which yields the bound $P_f(N) \le \exp(-C N^{1+
 \frac 2 {\al}}\log N)$.

\subsection{Leading constant: proof of Proposition \ref{cor: leading}}
In this section we construct a family of examples which establish Proposition~\ref{cor: leading}.
Fix $w\in L^1([-\pi,\pi])$ 
obeying $a< w(\lm) < b$ for all $\lm \in [-\delta,\delta]$.  Denote by $f_1$ the GSP whose spectral measure has density $w$. By Theorem~\ref{Thm: main presentable}, there exist constants $C_1,C_2\in (0,\infty)$ such that
$
C_1 < - \frac 1 N \log P_{f_1}(N) < C_2, 
$
for all $N$ sufficiently  large.

Now, let $f_2$ be the GSP whose spectral measure has density $w(\lm) + \nu \ind_{\pm [\delta,\pi]}(\lm)$, where $\nu$ is to be defined shortly. 
Clearly, the spectral measures of $f_1$ and of $f_2$ coincide in $[-\delta, \delta]$.
Once again, by Theorem~\ref{Thm: main presentable}, we have
$C_4 < - \frac 1 N \log P_{f_2}(N) < C_3$, for some $C_3,C_4\in (0,\infty)$ and all large enough $N>0$. Our goal is to show that if $\nu$ is sufficiently large, then $C_2<C_4$. By Theorem \ref{thm: balance up} and Remark \ref{rmk: thm relax}, it holds that
\[
P_{f_1}(N) \le \Pro\left( Z > \ell \sqrt{N} \right)  +  2N \Pro\left( \nu \frac{\pi-\delta}{b}|Z|\le \ell  \right)^{c(\pi-\delta)N},
\]
where $\delta$, $b$ and $c$ are fixed. We proceed similarly to the proof of Proposition~\ref{cor: up}.
For $\ell=2 \sqrt{C_2}$, using Lemma~\ref{lem: tail}(a), we have
\[
\Pro\left( Z > \ell \sqrt{N} \right)  \le e^{-\frac 1 2 \ell^2 N}= e^{- 2C_2 N}.
\]
Next choose $\nu$ so large that, for sufficiently large $N$, 
\[
2N\, \Pro\left( \nu \frac{\pi-\delta}{b}|Z|\le \ell  \right)^{c(\pi-\delta)N} \le e^{- 2C_2 N}.
\]
This implies 
\[
C_1 < - \frac 1 N \log P_{f_1}(N) < C_2 < 2C_2 -\frac{\log 2}{N} \le  - \frac 1 N \log P_{f_2}(N) < C_3,
\]
as required.

\section{Proof of the analytic theorem}\label{sec: anal}
In this section we prove Theorem~\ref{thm: anal}, first for continuous and then for discrete time.

\subsection{Continuous time}
Here we always assume $I=[-1,1]$ (the general case follows by scaling).
We will prove the following slightly stronger version of Theorem \ref{thm: anal} in continuous time.
Notice that the constant $\pm \tfrac 9{20}$ in the integral limits is improved to be $\pm \tfrac 9{10}$.

\begin{prop}\label{thm: anal contra}
Let $f$ be $k$-times differentiable with $\inf_I f^{(k)}>0$. Suppose that
$ \int_{-0.9}^{0.9} f^{(k)}\ge k!$.
Then $\sup_{I} |f| \ge \frac 1{c^k}$, where $c>0$ is a universal constant.
\end{prop}

The proof will use Chebyshev's polynomials of the first kind, defined through
\[
T_k(x) = \cos (k \arccos x), \quad x\in[-1,1].
\]
Fix $k\in \N$, and let $x_j=\cos \left(\frac {(k-j)\pi}{k}\right)$, $j=0,1,\dots,k$ be the $k+1$ extremal points for $T_k(x)$ (hereafter called ``Chebyshev extrema of order $k$''). Notice that $-1=x_0<x_1<\dots < x_k=1$ and $T_k(x_j)=(-1)^{k-j}$.

Let the $L^\infty(\{x_j\})$-norm of a function $f:\{x_j\}\to\R$ be $\norm f = \max{|f(x_j)|}$.
A classical property of Chebyshev polynomials is the following.

\begin{clm}\label{clm: min norm}
For $k\in\N$, the polynomial $a 2^{1-k} T_k$ has minimal $L^\infty(\{x_j\})$-norm among all polynomials of degree $k$ and leading coefficient $a>0$.
The value of this norm is $a 2^{1-k}$.
\end{clm}

\begin{proof}
Without loss of generality, we may assume that $a=1$.
Suppose $P_k$ was a monic polynomial of degree $k$ with {$\norm {P_k}< 2^{1-k}$}.
Recall that $2^{1-k}T_k(x_j)=(-1)^{k-j} 2^{1-k}$ on all points $x_j$.
Thus the difference $w(x) = 2^{1-k}T_k(x)-P_k(x)$ alternates signs when evaluated on the points $\{x_j\}$.
By the intermediate value theorem, $w$ has at least $k$ roots. But this is impossible, since $w$ is a polynomial of degree $\le k-1$.
\end{proof}

Denote by $f[x_0,x_1,\dots,x_k]$
the leading coefficient of the unique degree-$k$ polynomial which interpolates $f$ at the points $x_0,x_1,\dots, x_{k}$. The next claim states that, under similar conditions to those of Proposition \ref{thm: anal contra}, this leading coefficient cannot be too small.

\begin{clm}\label{clm: leading con}
Suppose that $f:I\to\R$ is $(k-1)$-times differentiable, and that $f^{(k-1)}$ is piecewise differentiable with $f^{(k)}>0$ and 
$ \int_{-0.9}^{0.9} f^{(k)}\ge k!$.
Then:
\[
f[x_0,x_1,\dots,x_k] \ge M^{-k},
\]
where $M>0$ is a universal constant.
\end{clm}

First let us see how {Claim \ref{clm: leading con} may be used} to conclude the proof of Proposition \ref{thm: anal contra}.
\begin{proof}[Proof of Proposition~\ref{thm: anal contra}]
Let $f:I\to\R$ be as in the {premise}, that is, $f^{(k)}>0$ for all $x\in I$ and $\int_{-0.9}^{0.9} f^{(k)}\ge k!$.
Let $P_k$ be the interpolation polynomial of $f$ at the Chebyshev extremal points $\{x_j\}_{j=0}^k$, and let
$a=f[x_0,\dots,x_k]$ denote the leading coefficient of $P_k$.
By Claim~\ref{clm: leading con}, we have $|a|\ge M^{-k}$ (where $M>0$ is universal), so that by Claim~\ref{clm: min norm} we deduce that
\begin{align*}
\sup_I |f| &\ge \max_{0\le j\le k} |f(x_j)| =\max_{0\le j\le k} |P_k(x_j)|\ge |a| 2^{1-k}\ge (2M)^{-k}.
\end{align*}
\end{proof}

It remains to prove Claim \ref{clm: leading con}.  To do so, we apply the following standard result from interpolation theory, known as the \emph{Hermite-Genocchi formula} (see e.g. \cite{Sch}*{Thm 4.2.3}): 

\begin{lem}[The Hermite-Genocchi formula]\label{lem: HG}
Let $\{x_j\}_{j=0}^{k}\subset I$ be distinct points, and let $f: I\to \R$ be {$(k-1)$-times differentiable with $f^{(k-1)}$ being piecewise differentiable}. Then:
\[
f[x_0,x_1,\dots,x_{k}] = \int_{\Sigma_{k}} f^{(k)}\Big( t_0 x_0 +t_1x_1 +\dots t_{k} x_{k}\Big) d\sigma,
\]
where
\begin{equation}\label{eq: simplex}
\Sigma_{k} = \{(t_0,\dots,t_{k}): \; \forall j\ t_j\ge 0  \text{ and } \sum_{j=0}^{k} t_j=1 \},
\end{equation}
and $d\sigma$ is the induced volume form on $\Sigma_k$.
\end{lem}

Notice that $\Sigma_{k}$ is {a} $k$-dimensional simplex (embedded in $\R^{k+1}$), thus $\text{Vol}_{k}(\Sigma_{k})=1/k!$. The following lemma enables us to bound efficiently the {integral} which appears in the Hermite-Genocchi formula.

\begin{lem}\label{lem: bound}
Fix $k\in\N$, and let $\{x_j\}_{j=0}^{k}\subset I$ be the Chebyshev extrema points. Write $\Sigma_{k}$ for the simplex in \eqref{eq: simplex}.
Then there exists a non-negative continuous function $g_{k}\in C(I)$ so that for any $F\in L^1(I)$ we have
\begin{equation}\label{eq: for g}
 \int_{\Sigma_{k}}F \Big( t_0 x_0 +t_1x_1 +\dots t_{k} x_{k}\Big) d\sigma = \int_{-1}^{1} F(s)g_{k}(s) ds.
\end{equation}
Moreover, there exists $L>1$ such that for all $k\in\N$ and all $|s|\le 0.9$, we have $g_{k}(s)\ge \frac 1 {k! L^{k}}$.
\end{lem}

\begin{proof}
If $k=1$ the statement is true with $g_1(s)\equiv\tfrac {1}{2}$. For $k>1$, the function
\[
g_{k}(s) = \text{Vol}_{k-1} \Big(\Big\{(t_0,t_1,\dots,t_{k}): \: \sum_j t_j x_j=s,\, \sum_j t_j=1,\, \forall i \ t_i\ge 0 \Big\}\Big)
\]
satisfies \eqref{eq: for g} with any $F\in L^1 (I)$, due to a change of variables (and Fubini's theorem).
It is clear that $g_k$ is non-negative and continuous on $I$.

For $s\in I=[-1,1]$ define the set
\[
A_s = \left\{(t_0,\dots, t_{k}): \;  \sum_{j=0}^{k} t_j=1, \ \sum_{j=0}^{k} t_j x_j =s,\  (t_1,\dots, t_{k-1})\in\left[0,\frac 1 {20(k-1)}\right]^{k-1} \right\}
\]
Recalling that $x_0=-1$ and $x_k=1$ we notice that, given $(t_1,\dots,t_{k-1}) \in [0,1/(20(k-1))]^{k-1}$, there is a unique pair $(t_0,t_{k})\in\R^2$ that satisfies the independent linear equations
\begin{equation}\label{eq: sys}
\begin{cases}
\sum_{j=0}^{k} t_j=1,   \\
\sum_{j=0}^{k} t_j x_j =s,
\end{cases}
\iff
\begin{cases}
t_0+t_{k}=a\\
{-}t_0 {+} t_{k}=b,
\end{cases}
\end{equation}
where $a,b$ are numbers (depending on $t_1,\dots t_{k-1}$).
Therefore, $A_s$ is a $(k-1)$-dimensional manifold (embedded in $\R^{k+1}$), which may be viewed as the graph of a linear function on the domain
$\left[0,\tfrac 1 {20(k-1)}\right]^{k-1}$. This implies that $\text{Vol}_{k-1}(A_s) \ge \left(\frac 1{20(k-1)}\right)^{k-1} \ge  \frac 1 {k! L^{k}}$ for some $L>0$.

It remains to check that if $|s|\le 0.9$, for any $(t_1,\dots,t_{k-1}) \in\left[0,\frac 1 {20(k-1)}\right]^{k-1}$ the solution to \eqref{eq: sys} obeys $t_0\ge 0, \ t_{k}\ge 0$ (allowing us to conclude that $g_{k}(s) \ge \text{Vol}_{k-1}(A_s)$). Solving explicitly we have

\[
t_0 =
\frac{1}{2} \Big(1-s-\sum_{j=1}^{k-1}t_j (1-x_j)\Big) ,
\]
so that $t_0\ge 0$ if and only if $\sum_{j=1}^{k-1}t_j (1-x_j) \le 1-s$. We verify this as follows:
\[
\sum_{j=1}^{k-1}t_j (1-x_j) \le \frac {2} {20(k-1)}\cdot(k-1) = \frac 1{10} \le 1-s ,
\]
where we used that $|s|\le 0.9$. The verification of $t_{k}\ge 0$ is obtained in a symmetric way.
\end{proof}

Now we return to prove Claim \ref{clm: leading con}.

\begin{proof}[Proof of Claim \ref{clm: leading con}]
By using Lemma~\ref{lem: HG} and then applying Lemma~\ref{lem: bound} with $F=f^{(k)}$, we get that:
\begin{align*}
f[x_0,x_1,\dots,x_k] &=\int_{\Sigma_k} f^{(k)}\Big( t_0 x_0 +t_1x_1 +\dots t_k x_k\Big) dt
\\& \ge \frac 1 {k! L^k} \int_{-0.9}^{0.9} f^{(k)}(s) ds
\\ & \ge \frac 1{L^k}.
\end{align*}
where $L>0$ is a universal constant.
\end{proof}

\subsection{Discrete time}

In this section we prove Theorem \ref{thm: anal} in discrete time using an interpolation method.
We denote the discrete derivative operator by $\Delta$ that is,
\begin{equation}\label{eq: Delta}
(\Delta f)(n) = f(n+1)-f(n).
\end{equation}
When applied iteratively $k$ times we write $\Delta^k$.

A somewhat stronger version of Theorem \ref{thm: anal} over $\Z$ (similar to Proposition \ref{thm: anal contra}) is the following.
\begin{prop}\label{thm: anal contra dis}
Let $k,N\in \N$ be two numbers which satisfy $N\ge k$. Let $f: \Z \to \R$ be such that {$\Delta^{(k)}f(n) >0$} for every $n\in {[-N,N]}$.
Denote $I_N=[-\tfrac 9{10}N, \tfrac{9}{10}N]\cap\Z$ and suppose that
$ \frac 1 {N} \sum_{{n\in I_N} } f^{(k)}(n) \ge{\frac{k!}{N^k}}$.
Then $\displaystyle\sup_{[-2N,2N]} |f| \ge \frac 1{c^k}$, where $c>0$ is a universal constant.
\end{prop}

\begin{proof}

{Let $B_0(x) =\ind_{\big[-\tfrac 1 2,\tfrac 1 2\big)}(x)$ and  $B_k(x) = (B_0) ^{*k}(x)$, where $(B_0) ^{*k}$ is the convolution of $B_0$ with itself $k$ times (e.g. $B_1 = B_0*B_0$).} Define the continuous time function:
\[
F(x) = \sum_{n\in\Z} f(n) B_{k+1}(x-n).
\]
 Observe that $B_k$ takes values in $[0,1]$, is finitely supported, is a piecewise degree-$k$ polynomial, and $\sum_{n\in \Z} B_k(x-n)=1$ for all $x\in\R$.
The derivative of $F$ is given by the following relation (see \cite{dB}*{eq. (10.3)}):
\[
F'(x)=\frac{d}{dx}\left( \sum_{n\in\Z} f(n) B_{k+1}(x-n)\right) = \sum_{n\in \Z} (\Delta f)(n) B_k (x-n),
\]
where $\Delta$ is the discrete derivative (recall \eqref{eq: Delta}).
By repeating this $k$ times we get:
\[
F^{(k)}(x) =\sum_{n\in \Z} \D^{k} f(n) B_1 (x-n).
\]

Since $\D^k f(n)>0$ for all $n\in{[-N,N]}\cap \Z$, {we have} $F^{(k)}(x) > 0$ for all $x\in [-N+1,N-1]$.
Moreover, we have,
\[ \frac 1 N \int_{{-0.9N}}^{{0.9N}} F^{(k)}\ge  \frac 1 {2N} \sum_{ n\in [-0.9N,0.9N]\cap\Z} \D^k f(n) \ge  \frac 1 2 \frac{k!}{N^k}.\]
Therefore, by Proposition \ref{thm: anal contra} we deduce that
\begin{align*}
\sup_{n\in [-2N,2N]\cap\Z} f(n)\ge
 \sup_{n\in {[-N-k,N+k]}\cap \Z} f(n)\geq \sup_{x\in[-N+1,N-1] } F(x) \ge \frac{1}{c^k},
\end{align*}
for some universal number $c$, as required. This concludes the proof.
\end{proof}

\vspace{15pt}
{\sc{Acknowledgments:}} 

N.F. acknowledges the supports of NSF postdoctoral fellowship MSPRF-1503094 held at Stanford.
O.F. acknowledges the support of NSF grant DMS-1613091 while holding a postdoctoral position at Stanford.
S.N. acknowledges the support of NSF grant DMS-1600726.

We are grateful to Mikhail Sodin, for introducing us to the project and encouraging us to work together.
We thank Amir Dembo, for useful discussions, advice and support.
We thank Sasha Sodin for pointing out that Chebyshev polynomials may be useful for proving Theorem \ref{thm: anal}, and for referring us to \cite{Bern} and \cite{Ganz}.
Finally, we thank the reviewers for their useful suggestions and comments and especially for suggesting us to include Proposition~\ref{cor: leading}.


\bigskip

\bigskip


  N.~Feldheim, \textsc{
Department of Mathematics, Bar-Ilan University 
}\par\nopagebreak
  \textit{E-mail address}: \texttt{
naomi.feldheim@biu.ac.il
}

\medskip

O.~Feldheim, \textsc{Einstein Institute for Mathematics, the Hebrew University}
\par\nopagebreak
  \textit{E-mail address}: \texttt{
ohad.feldheim@mail.huji.ac.il
}

\medskip

S.~Nitzan, \textsc{School of Mathematics, Georgia Institute of Technology}
\par\nopagebreak
  \textit{E-mail address}: \texttt{
shahaf.nitzan@math.gatech.edu
}

\end{document}